\documentclass[12pt]{scrartcl}

\usepackage[utf8]{inputenc}
\usepackage[T1]{fontenc}
\usepackage{babel}

\usepackage{amscd, amssymb, mathtools, amsthm, amsmath, mathrsfs}
\usepackage{enumerate}

\usepackage[osf,sc]{mathpazo}

\usepackage[scaled=0.90]{helvet}

\usepackage[scaled=0.85]{beramono}

\usepackage{hyperref}

\usepackage[page]{appendix}

\newtheorem{proposition}{Proposition}[section]

\newtheorem{definition}[proposition]{Definition}

\newtheorem{cor}[proposition]{Corollary}

\newtheorem*{rem}{Remark}
\newtheorem{remark}[proposition]{Remark}

\newtheorem{notation}{Notation}

\usepackage{tikz-cd}
\usetikzlibrary{cd}

\usepackage[parfill]{parskip}

\providecommand{\keywords}[1]{\textbf{\textit{Keywords---}} #1}

\providecommand{\MSC}[1]{\textbf{2010 MSC:} #1}

\usepackage{authblk}

\def\spec{\text{Spec}\,R}

\makeindex

\usepackage{isabelle,isabellesym}

\begin{document}

\title{Simple Type Theory is not too Simple}
\subtitle{Grothendieck's Schemes without Dependent Types}
\author[1]{Anthony Bordg\thanks{apdb3@cam.ac.uk}}
\author[2]{Lawrence Paulson}
\author[3]{Wenda Li}
\affil[1-3]{University of Cambridge}
\date{}
\maketitle

\begin{abstract}
Church's simple type theory is often deemed too simple for elaborate mathematical constructions. In particular, doubts were raised whether schemes could be formalized in this setting and a challenge was issued. Schemes are sophisticated mathematical objects in algebraic geometry introduced by Alexander Grothendieck in 1960. In this article we report on a successful formalization of schemes in the simple type theory of the proof assistant Isabelle/HOL, and we discuss the design choices which make this work possible. We show in the particular case of schemes how the powerful dependent types of Coq or Lean can be traded for a minimalist apparatus called locales. 
\end{abstract}

\MSC{14A15, 14-04, 03B35, 68V20}


\section{Introduction}
\label{sec:intro}

\subsection*{A Challenge Accepted and Met}

Proof assistants have made impressive progress on the formalisation of algebra \cite{oddorderthm}. In 2003, Laurent Chicli, as the topic of his PhD thesis \cite{chiclithesis}, formalized sheaves of rings and affine schemes in the proof assistant Coq \cite{coqrefmanual}, which is based on a powerful dependent type theory known as the \emph{Calculus of Inductive Constructions} \cite{coquand1986calculus, coquand1988inductively}. However, Chicli's Coq code certainly became deprecated long ago. More recently, in 2019 Grothendieck's schemes \cite{GrothendieckEGAI} have been formalized by a group of six mathematicians led by Kevin Buzzard \cite{schemesinLean}, with additional insights from Reid Barton and Mario Carneiro, using the brave new Lean theorem prover, which is based on a similar dependent type theory \cite{de2015lean}.
For the reader unfamiliar with schemes we will quote Kevin Buzzard.
\begin{quote}
	Schemes are the fundamental objects of study in algebraic geometry. They were discovered (in their current form) by Grothendieck in the late 1950s and early 1960s and they revolutionised the theory of algebraic geometry.\footnote{\url{https://github.com/leanprover-community/mathlib/issues/26}} 
\end{quote}
and
\begin{quote}
	A scheme is a mathematical object whose definition and basic properties are usually taught at MSc or early PhD level in a typical mathematics department.
\end{quote}	
Kevin Buzzard in 2017 proposed the formalization of schemes 
\begin{quote}
	as basically a challenge to see if Lean can handle such a complex definition.\footnote{\url{https://github.com/leanprover-community/mathlib/issues/26}} 
\end{quote}
We should note that the Isabelle proof assistant is based on a much more minimalist type theory, known as simple type theory. As a consequence, many users of dependent type theories see Isabelle's type theory as a far less expressive system for formalizing advanced mathematics.  These doubts have been expressed by Kevin Buzzard on his blog:
\begin{quote}
	What can Isabelle/HOL actually do before it breaks? Nobody knows. [\dots] But do you people want to attract “working mathematicians”? Then where are the schemes? Can your system even do schemes? I don’t know. Does anyone know? If it cannot then this would be very valuable to know because it will help mathematician early adopters to make an informed decision about which system to use.\footnote{\url{https://xenaproject.wordpress.com/2020/02/09/where-is-the-fashionable-mathematics/}}
\end{quote}

Even the status of sheaves of rings, a prerequisite for schemes, was unclear:
\begin{quote}
	I am not sure that it is even possible to write this code in Isabelle/HOL in such a way that it will run in finite time, [\dots], and a sheaf is a dependent type, and your clever HOL workarounds will not let you use typeclasses [\dots]\footnote{\url{https://xenaproject.wordpress.com/2020/02/09/where-is-the-fashionable-mathematics/}. See the section entitled \textit{Isabelle/HOL users}.}
\end{quote}

In this article we present a formalization of sheaves of rings and schemes in Isa\-belle/HOL reaching the benchmark result set by \cite{schemesinLean}: an affine scheme is a scheme.

\subsection*{Towards a Discipline of Formalizing}

To achieve this result, we tried to bridge the gap between set theory and type theory within the simple type theory of Isabelle/HOL, while forgoing the extension of Isabelle's logic with a new rule for type definition as proposed by Kun\v{c}ar and Popescu in their types-to-sets framework \cite{typestosets}. We avoided Isabelle's type classes and type declarations whenever possible and relied on Isabelle's module system called \textit{locales}, whose modern incarnation appeared in 2006 \cite{ballarin2006interpretation}. Locales pervade our development and we exemplify their use in topology, abstract algebra and algebraic geometry. In particular, the present work makes a triple contribution to Isabelle/HOL: a first building block towards a new topology library, a new building block towards a library for abstract algebra and finally a formalization of schemes. The formalized material has a size of 7300 lines of code \cite{Grothendieck_Schemes-AFP}, including material in commutative algebra (prime and maximal ideals, localization of a ring) and our partial reconstruction of the topology library (down to the notion of a topological space).
Our experiment meets the challenge of formalizing for the first time schemes in simple type theory and by doing so it demonstrates that intricate dependencies of mathematical structures can be managed in this formal logical system. In Section \ref{sec-conclusion} the current limitations of locales are addressed and we point out possible improvements to make formal statements handled with locales more concise. Our approach through locales allows for intricate definitions as well as proofs where the efficient proof automation of Isabelle/HOL kicks in. We hope this work will help mathematicians interested in the formalization of mathematics to evaluate the formal systems available to them.

\section{A World Without Dependent Types: Problems and solutions}
\label{sec:isabelle}

\subsection*{Isabelle/HOL and Simple Type Theory}

Isabelle/HOL~\cite{isa-tutorial} is a proof assistant for higher-order logic---Church's simple type theory~\cite{church40}---and is therefore based on the typed $\lambda$-calculus with boolean and function types. A boolean-valued function is a predicate, which yields a simple typed set theory that is expressive enough to express abstract mathematics, as we demonstrate below. 

In Church's original conception, the sole purpose of types is to prevent inconsistency, all collections being represented as sets. People today, influenced by programming languages, prefer to give types a more dominant role. They expect to have numeric types like \isa{int}, \isa{nat}, \isa{real}, \isa{complex} and for types to play a major role in logical reasoning. For example, an expression like \isa{x+y} should be interpreted with respect to the types of \isa{x} and~\isa{y}.
Isabelle/HOL's \emph{type classes} \cite{wenzel-type} support this sort of type-based disambiguation in a strong sense, so that \isa{x+y} refers to the correct instance of addition.
Type classes work for such trivialities as the basic laws for addition and multiplication (with the help of type classes for groups and rings) and for more advanced topological properties~\cite{hoelzl-filters}. A new type, once proved to be a topological space, instantly inherits such concepts as limits and continuity along with all theorems proved about them.

The drawback of type classes is simply that they refer to types. Mathematical constructions are generally too complex to formalise as types. Even the carrier of a group can be a complex object. Moving to a stronger type system, as done in Coq or Lean, brings issues of its own. Our solution is to define what we need using the naive set theory that comes with higher-order logic, through the mechanism of locales.%
\footnote{Not to be confused with the locales of point-free topology!}

\subsection*{Locales}
\label{subsec:locales}

A \emph{locale}~\cite{ballarin-locales-module} represents an Isabelle/HOL proof context. It can take parameters and instances of other locales. It can declare components constrained by assumptions. Thus they correspond directly to the mathematical practice of defining a monoid, say, as a tuple $(M,\cdot ,1)$ satisfying the obvious properties. 
A locale can inherit multiple instances, as when we define a monoid homomorphism with respect to two monoids.

Logically, a locale is nothing but a predicate; their power comes from Isabelle's mechanisms for creating, managing and using locale hierarchies.
If we work in a particular locale, we can refer to its components and assumptions. But we can also prove membership of a locale, which means to exhibit particular constructions---say a purported monoid carrier along with its multiplication and identity---and prove that they satisfy the locale's assumptions. \emph{Sublocale} declarations provide a means of proving inclusions between locales, for instance to prove once and for all that every submonoid is in particular a monoid.
We found that locales handled the tangle of definitions building up to schemes faithfully.

However, the right way to formalise advanced algebra is not obvious. During the formal development the user has to decide how mathematical structures should be best organized in the proof assistant so that they can be used mechanically in an efficient way.
It is well-known in the Isabelle community that the main algebra library needs to be rebuilt from scratch. In fact, there are many libraries for algebra in Isabelle: \textit{HOL-Algebra}\footnote{\url{https://isabelle.in.tum.de/dist/library/HOL/HOL-Algebra/}} from the previous millennium (1999), the more recent \textit{HOL-Computational\_Algebra}\footnote{\url{https://isabelle.in.tum.de/dist/library/HOL/HOL-Computational_Algebra/}}, and moreover some entries in the Archive of Formal Proofs\footnote{https://www.isa-afp.org/} like \textit{Groups, Rings and Modules} \cite{Group-Ring-Module-AFP} and \textit{Vector Spaces} \cite{VectorSpace-AFP} among others. As a result, there is a lot of overlapping material, which makes the status of algebra confusing and unclear for the newcomer in the Isabelle world. And while these formalizations constitute an impressive body of work, parts of it are deprecated code not using locales or the Isar language, an additional layer of vernacular which allows for structured proofs making them more legible and easier to maintain. Despite these problems, there are noteworthy achievements like the formalisation of the algebraic closure of a field \cite{de2020algebraically}.

Even using locales, we can still go wrong. The root of the problems in HOL-Algebra is the desire to refer to a structure such as a group with its components using a single variable, as a record data structure. The new (at the time) extensible records seemed perfect for the task.
But they led to some peculiarities: notably, the locale \emph{abelian\_group} (in \textit{Theory Ring}) which presents the odd twist of requiring a ring structure. The lack of multiple inheritance for records seems to have required the awkward use of the ring record for abelian groups. However, Clemens Ballarin recently conducted an experiment \cite{ballarin-exploring} showing that  locales, without records, allow for a smooth handling of basic algebraic structures in Isabelle such as monoids, groups and rings. We decided to base our formal development on this experiment.  
In the next section we introduce schemes and their formal counterparts in simple type theory, replacing the dependent types of Buzzard et al.\  \cite{schemesinLean} by Isabelle's locales.

\section{Schemes in Isabelle/HOL}
\label{sec:schemes}

\subsection*{Elements of Topology}

We seized the opportunity of formalizing schemes to build a new topology library despite the two existing formalizations of topology in Isabelle/HOL, namely \textit{HOL-Topological\_Spaces.thy}\footnote{\url{https://isabelle.in.tum.de/dist/library/HOL/HOL/Topological_Spaces.html}} and \textit{HOL-Analysis/Abstract\_Topology}\footnote{\url{https://isabelle.in.tum.de/dist/library/HOL/HOL-Analysis/Abstract_Topology.html}}. Our new topology library is entirely built on locales without using type classes or type declarations (via the \textit{typedef} command). In particular, our topological spaces have explicit carrier sets as part of their data instead of using \textit{UNIV}, the set of all elements of some type, or having to define it later as the union of all the open sets. 

\begin{isabelle}
	\isacommand{locale}\ topological\_space\ =\ \isanewline
	\ \ \isakeyword{fixes}\ S\ ::\ "'a\ set"\ \isakeyword{and}\ is\_open\ ::\ "'a\ set\ \isasymRightarrow \ bool"\isanewline
	\ \ \isakeyword{assumes}\ open\_space\ [simp,\ intro]:\ "is\_open\ S"\ \isanewline
	\ \ \ \ \isakeyword{and}\ open\_empty\ [simp,\ intro]:\ "is\_open\ \{\}"\ \isanewline
	\ \ \ \ \isakeyword{and}\ open\_imp\_subset:\ "is\_open\ U\ \isasymLongrightarrow \ U\ \isasymsubseteq \ S"\isanewline
	\ \ \ \ \isakeyword{and}\ open\_inter\ [intro]:\ "\isasymlbrakk is\_open\ U;\ is\_open\ V\isasymrbrakk \ \isasymLongrightarrow \ is\_open\ (U\ \isasyminter \ V)"\ \isanewline
	\ \ \ \ \isakeyword{and}\ open\_union\ [intro]:\ \isanewline
	\ \ \ \ \ \ \ \ \ \ \ \ \ \ \ \ "\isasymAnd F::('a\ set)\ set.\ (\isasymAnd x.\ x\ \isasymin \ F\ \isasymLongrightarrow \ is\_open\ x)\isanewline
	\ \ \ \ \ \ \ \ \ \ \ \ \ \ \ \ \ \ \ \ \ \ \ \ \ \ \ \ \ \ \ \ \ \ \ \ \ \ \ \ \ \ \ \ \ \ \ \ \ \isasymLongrightarrow \ is\_open\ (\isasymUnion x\isasymin F.\ x)"
\end{isabelle}

The token \textit{'a set} denotes the type of sets of elements of type \textit{'a}, this last type being the type of the elements of the carrier \textit{S} of our topological space.
Within the locale for topological spaces the axiom \textit{open\_imp\_subset} enforces that an open is automatically a subset of the ambient space. This is a welcome addition, since it avoids having to add this assumption in every lemma mentioning an open set among its assumptions, making the economy of dozens of trivial assumptions.

Then we define the topology generated by a family of subsets of a given set,

\begin{isabelle}
	\isacommand{inductive}\ generated\_topology\ ::\ "'a\ set\ \isasymRightarrow \ 'a\ set\ set\ \isasymRightarrow \ 'a\ set\ \isasymRightarrow \ bool"\ \isanewline
	\ \ \ \ \isakeyword{for}\ S\ ::\ "'a\ set"\ \isakeyword{and}\ B\ ::\ "'a\ set\ set"\isanewline
	\ \ \isakeyword{where}\isanewline
	\ \ \ \ UNIV:\ "generated\_topology\ S\ B\ S"\isanewline
	\ \ |\ Int:\ "generated\_topology\ S\ B\ (U\ \isasyminter \ V)"\ \isanewline
	\ \ \ \ \ \ \ \ \ \ \ \ \isakeyword{if}\ "generated\_topology\ S\ B\ U"\ \isakeyword{and}\ "generated\_topology\ S\ B\ V"\isanewline
	\ \ |\ UN:\ "generated\_topology\ S\ B\ (\isasymUnion K)"\ \isanewline
	\ \ \ \ \ \ \ \ \ \ \ \ \isakeyword{if}\ "(\isasymAnd U.\ U\ \isasymin \ K\ \isasymLongrightarrow \ generated\_topology\ S\ B\ U)"\isanewline
	\ \ |\ Basis:\ "generated\_topology\ S\ B\ b"\ \isakeyword{if}\ "b\ \isasymin \ B\ \isasymand \ b\ \isasymsubseteq \ S"
\end{isabelle}

where \textit{'a set set} denotes the type of sets of elements of type \textit{'a set}. The fact that the generated topology is indeed a topology is established as one of our first lemmas.  
\begin{isabelle}
	\isacommand{lemma}\ generated\_topology\_is\_topology:\isanewline
	\ \ \isakeyword{fixes}\ S::\ "'a\ set"\ \isakeyword{and}\ B::\ "'a\ set\ set"\isanewline
	\ \ \isakeyword{shows}\ "topological\_space\ S\ (generated\_topology\ S\ B)"
\end{isabelle}

Covers in topology are an interesting example, since they illustrate how locales can be nested like matryoshka dolls. 

\begin{isabelle}
	\isacommand{locale}\isamarkupfalse%
	\ cover{\isacharunderscore}{\kern0pt}of{\isacharunderscore}{\kern0pt}subset\ {\isacharequal}{\kern0pt}\isanewline
	\ \ \isakeyword{fixes}\ X{\isacharcolon}{\kern0pt}{\isacharcolon}{\kern0pt}\ {\isachardoublequoteopen}{\isacharprime}{\kern0pt}a\ set{\isachardoublequoteclose}\ \isakeyword{and}\ U{\isacharcolon}{\kern0pt}{\isacharcolon}{\kern0pt}\ {\isachardoublequoteopen}{\isacharprime}{\kern0pt}a\ set{\isachardoublequoteclose}\ \isanewline
	\ \ \ \ \isakeyword{and}\ index{\isacharcolon}{\kern0pt}{\isacharcolon}{\kern0pt}\ {\isachardoublequoteopen}real\ set{\isachardoublequoteclose}\ \isakeyword{and}\ cover{\isacharcolon}{\kern0pt}{\isacharcolon}{\kern0pt}\ {\isachardoublequoteopen}real\ {\isasymRightarrow}\ {\isacharprime}{\kern0pt}a\ set{\isachardoublequoteclose}\isanewline
	\ \ \isakeyword{assumes}\ {\isachardoublequoteopen}U\ {\isasymsubseteq}\ X{\isachardoublequoteclose}\ \isanewline
	\ \ \ \ \isakeyword{and}\ {\isachardoublequoteopen}{\isasymAnd}i{\isachardot}{\kern0pt}\ i\ {\isasymin}\ index\ {\isasymLongrightarrow}\ cover\ i\ {\isasymsubseteq}\ X{\isachardoublequoteclose}\isanewline
	\ \ \ \ \isakeyword{and}\ {\isachardoublequoteopen}U\ {\isasymsubseteq}\ {\isacharparenleft}{\kern0pt}{\isasymUnion}i{\isasymin}index{\isachardot}{\kern0pt}\ cover\ i{\isacharparenright}{\kern0pt}{\isachardoublequoteclose}\isanewline
	\isanewline
	\isacommand{locale}\isamarkupfalse%
	\ open{\isacharunderscore}{\kern0pt}cover{\isacharunderscore}{\kern0pt}of{\isacharunderscore}{\kern0pt}subset\ {\isacharequal}{\kern0pt}\ topological{\isacharunderscore}{\kern0pt}space\ X\ is{\isacharunderscore}{\kern0pt}open\ \isanewline
	\ \ \ \ \ \ \ \ \ \ {\isacharplus}{\kern0pt}\ cover{\isacharunderscore}{\kern0pt}of{\isacharunderscore}{\kern0pt}subset\ X\ U\ I\ C\ \isanewline
	\ \ \isakeyword{for}\ X\ \isakeyword{and}\ is{\isacharunderscore}{\kern0pt}open\ \isakeyword{and}\ U\ \isakeyword{and}\ I\ \isakeyword{and}\ C\ {\isacharplus}{\kern0pt}\isanewline
	\ \ \isakeyword{assumes}\ {\isachardoublequoteopen}{\isasymAnd}i{\isachardot}{\kern0pt}\ i{\isasymin}I\ {\isasymLongrightarrow}\ is{\isacharunderscore}{\kern0pt}open\ {\isacharparenleft}{\kern0pt}C\ i{\isacharparenright}{\kern0pt}{\isachardoublequoteclose}\isanewline
	\isanewline
	\isacommand{locale}\isamarkupfalse%
	\ open{\isacharunderscore}{\kern0pt}cover{\isacharunderscore}{\kern0pt}of{\isacharunderscore}{\kern0pt}open{\isacharunderscore}{\kern0pt}subset\ \isanewline
	\ \ \ \ \ \ \ \ \ \ \ \ {\isacharequal}{\kern0pt}\ open{\isacharunderscore}{\kern0pt}cover{\isacharunderscore}{\kern0pt}of{\isacharunderscore}{\kern0pt}subset\ X\ is{\isacharunderscore}{\kern0pt}open\ U\ I\ C\ \isanewline
	\ \ \isakeyword{for}\ X\ \isakeyword{and}\ is{\isacharunderscore}{\kern0pt}open\ \isakeyword{and}\ U\ \isakeyword{and}\ I\ \isakeyword{and}\ C\ {\isacharplus}{\kern0pt}\isanewline
	\ \ \isakeyword{assumes}\ {\isachardoublequoteopen}is{\isacharunderscore}{\kern0pt}open\ U{\isachardoublequoteclose}
\end{isabelle}

Above, each locale builds on the previous one to define eventually an open cover of an open subset. The token ``+'' is used to declare a mathematical structure which was previously introduced and also to add possibly a list of assumptions.
Our library includes also basic elements like homeomorphisms and the subspace topology induced on a subset.

\subsection*{Elements of Algebraic Geometry}

Our presentation follows the classic graduate textbook of Hartshorne \cite{hartshorne}.
For the remainder of this text fix a ring $R$. The reader can safely assume that $R$ is commutative.

\subsubsection*{The Zariski Topology}

The definitions of a ring $R$  and of an ideal of $R$ were introduced in \cite{ballarin-exploring}. Note that in Isabelle we do not assume from the start that $R$ is commutative, since we strive for generality. In formal mathematics, adding an axiom later is easier than removing one!

\begin{definition}[ideal]
	An ideal $\mathfrak{a}$ of $R$ is a subset of $R$ which is an additive subgroup such that $R\mathfrak{a} = \mathfrak{a}$.
\end{definition}

\begin{isabelle}
\isacommand{locale}\ ideal\ =\ subgroup\_of\_additive\_group\_of\_ring\ +\isanewline
\ \ \isakeyword{assumes}\ ideal:\ "\isasymlbrakk \ a\ \isasymin \ R;\ b\ \isasymin \ I\ \isasymrbrakk \ \isasymLongrightarrow \ a\ \isasymcdot \ b\ \isasymin \ I"\ \isanewline
\ \ \ \ \ \ \ \ \ \ \ \ \ \ \ \ \ "\isasymlbrakk \ a\ \isasymin \ R;\ b\ \isasymin \ I\ \isasymrbrakk \ \isasymLongrightarrow \ b\ \isasymcdot \ a\ \isasymin \ I"
\end{isabelle}

The token \textit{subgroup\_of\_additive\_group\_of\_ring} is the way to introduce in Isabelle a ring $R$ and an additive subgroup of $R$ using Isabelle's locale mechanism.
Note that in the code above the ideal is denoted $I$ while gothic letters are sometimes used for denoting ideals. In our code we use both depending on the situation.
One easily checks that $\lbrace 0 \rbrace$ and $R$ are ideals of $R$. In our formalization this translates as follows.

\begin{isabelle}
\isacommand{lemma}\ (\isakeyword{in}\ comm\_ring)\ ideal\_R\_R:\ "ideal\ R\ R\ (+)\ (\isasymcdot )\ \isasymzero \ \isasymone "
\end{isabelle}

\begin{isabelle}
\isacommand{lemma}\ (\isakeyword{in}\ comm\_ring)\ ideal\_0\_R:\ "ideal\ \{\isasymzero \}\ R\ (+)\ (\isasymcdot )\ \isasymzero \ \isasymone "
\end{isabelle}

The token ``\textbf{in} \textit{comm\_ring}'' allows to open a context where a commutative ring, here denoted $R$, has been previously declared in Isabelle.
Given $\mathfrak{a}$ and $\mathfrak{b}$ two ideals of $R$, one defines $\mathfrak{a} \mathfrak{b}$ to be the ideal whose underlying set is 
\[
	\lbrace a_1 b_1 + \dots + a_n b_n \mid a_i \in \mathfrak{a}, b_i \in \mathfrak{b}, n \in \mathbb{N} \rbrace.
\]

Actually, $\mathfrak{a} \mathfrak{b}$ is the smallest ideal generated by the set 
	\[
	\lbrace a b \mid a \in \mathfrak{a}, b \in \mathfrak{b} \rbrace \, .
	\]

\begin{isabelle}
\isacommand{definition}\ (\isakeyword{in}\ comm\_ring)\ ideal\_gen\_by\_prod\ \isanewline
\ \ \ \ \ \ \ \ \ \ \ \ \ \ \ \ \ \ \ \ \ \ \ \ \ \ \ \ \ \ ::\ "'a\ set\ \isasymRightarrow \ 'a\ set\ \isasymRightarrow \ 'a\ set"\isanewline
\ \ \isakeyword{where}\ "ideal\_gen\_by\_prod\ \isasymaa \ \isasymbb \ \isasymequiv \ additive.subgroup\_generated\ \isanewline
\ \ \ \ \ \ \ \ \ \ \ \ \ \ \ \ \ \ \ \ \ \ \ \ \{x.\ \isasymexists a\ b.\ x\ =\ a\ \isasymcdot \ b\ \isasymand \ a\ \isasymin \ \isasymaa \ \isasymand \ b\ \isasymin \ \isasymbb \}"
\end{isabelle}

One easily checks $\mathfrak{a} \mathfrak{b}$ is an ideal.

\begin{isabelle}
\isacommand{lemma}\ (\isakeyword{in}\ comm\_ring)\ ideal\_subgroup\_generated:\isanewline
\ \ \isakeyword{assumes}\ "ideal\ \isasymaa \ R\ (+)\ (\isasymcdot )\ \isasymzero \ \isasymone "\ \isakeyword{and}\ "ideal\ \isasymbb \ R\ (+)\ (\isasymcdot )\ \isasymzero \ \isasymone "\isanewline
\ \ \isakeyword{shows}\ "ideal\ (ideal\_gen\_by\_prod\ \isasymaa \ \isasymbb )\ R\ (+)\ (\isasymcdot )\ \isasymzero \ \isasymone "
\end{isabelle}

This ideal can be characterized as the intersection of all ideals of $R$ containing the set
	\[
	\lbrace a b \mid a \in \mathfrak{a}, b \in \mathfrak{b} \rbrace .
	\]
	
\begin{isabelle}
\isacommand{lemma}\ (\isakeyword{in}\ comm\_ring)\ ideal\_gen\_by\_prod\_is\_inter:\isanewline
\ \ \isakeyword{assumes}\ "ideal\ \isasymaa \ R\ (+)\ (\isasymcdot )\ \isasymzero \ \isasymone "\ \isakeyword{and}\ "ideal\ \isasymbb \ R\ (+)\ (\isasymcdot )\ \isasymzero \ \isasymone "\isanewline
\ \ \isakeyword{shows}\ "ideal\_gen\_by\_prod\ \isasymaa \ \isasymbb \ =\ \isasymInter \ \{I.\ ideal\ I\ R\ (+)\ (\isasymcdot )\ \isasymzero \ \isasymone \ \isanewline
\ \ \ \ \ \ \ \ \ \ \ \ \ \ \ \ \ \ \ \ \ \ \ \ \ \ \isasymand \ \{a\ \isasymcdot \ b\ |a\ b.\ a\ \isasymin \ \isasymaa \ \isasymand \ b\ \isasymin \ \isasymbb \}\ \isasymsubseteq \ I\}"
\end{isabelle}

Given $\lbrace \mathfrak{a}_i \rbrace$ a family of ideals of $R$ indexed by an arbitrary set $I$, one defines the set $\displaystyle \sum_{i \in I} \mathfrak{a}_i$ to be the set 
	\[
	\lbrace a_1 + \dots + a_n \mid n \in \mathbb{N}, a_k \in \mathfrak{a}_{i_k} \rbrace .
	\]
of all finite sums of elements belonging to some ideals of the family. 
The set $\displaystyle \sum_{i \in I} \mathfrak{a}_i$, seen as an additive subgroup, is again an ideal of $R$.	

The so-called prime ideals are an important class of ideals.

\begin{definition}[prime ideal]
	A prime ideal is an ideal $\mathfrak{p} \neq R$ such that $x y \in \mathfrak{p}$ implies $x \in \mathfrak{p}$ or $y \in \mathfrak{p}$ for every elements $x$ and $y$ in $R$.
\end{definition}

\begin{isabelle}
\isacommand{locale}\ pr\_ideal\ =\ comm:comm\_ring\ R\ "(+)"\ "(\isasymcdot )"\ "\isasymzero "\ "\isasymone "\ \isanewline
\ \ \ \ \ \ \ \ \ \ \ \ \ \ \ \ \ \ \ \ \ \ \ \ +\ ideal\ I\ R\ "(+)"\ "(\isasymcdot )"\ "\isasymzero "\ "\isasymone "\isanewline
\ \ \ \ \isakeyword{for}\ R\ \isakeyword{and}\ I\ \isakeyword{and}\ addition\ (\isakeyword{infixl}\ "+"\ 65)\ \isakeyword{and}\ multiplication\ \isanewline
\ \ \ \ \ \ \ (\isakeyword{infixl}\ "\isasymcdot "\ 70)\ \isakeyword{and}\ zero\ ("\isasymzero ")\ \isakeyword{and}\ unit\ ("\isasymone ")\isanewline
\ \ \ +\ \isakeyword{assumes}\ carrier\_neq:\ "I\ \isasymnoteq \ R"\ \isanewline
\ \ \ \ \ \ \ \ \ \ \isakeyword{and}\ absorbent:\ "\isasymlbrakk x\ \isasymin \ R;\ y\ \isasymin \ R\isasymrbrakk \ \isasymLongrightarrow \ (x\ \isasymcdot \ y\ \isasymin \ I)\ \isanewline
\ \ \ \ \ \ \ \ \ \ \ \ \ \ \ \ \ \ \ \ \ \ \ \ \ \ \ \ \ \ \ \ \ \ \ \ \ \ \ \ \ \ \isasymLongrightarrow \ (x\ \isasymin \ I\ \isasymor \ y\ \isasymin \ I)"
\end{isabelle}

Our locale \textit{pr\_ideal} builds upon the locale \textit{ideal} for ideals (introduced previously in \cite{ballarin-exploring}) which does not assume the commutativity of the ring $R$. As a consequence, in order to assume the commutativity of the ring, we have to declare afresh the parameters of the locale instead of simply writing
\begin{isabelle}
\isacommand{locale}\ pr\_ideal\ =\ ideal\ +\isanewline
\ \ \ \  \isakeyword{assumes}\ carrier\_neq:\ "I\ \isasymnoteq \ R"\ \isanewline
\ \ \ \ \ \ \ \ \ \ \isakeyword{and}\ absorbent:\ "\isasymlbrakk x\ \isasymin \ R;\ y\ \isasymin \ R\isasymrbrakk \ \isasymLongrightarrow \ (x\ \isasymcdot \ y\ \isasymin \ I)\ \isanewline
\ \ \ \ \ \ \ \ \ \ \ \ \ \ \ \ \ \ \ \ \ \ \ \ \ \ \ \ \ \ \ \ \ \ \ \ \ \ \ \ \ \ \isasymLongrightarrow \ (x\ \isasymin \ I\ \isasymor \ y\ \isasymin \ I)"
\end{isabelle}
as our locale, although this last locale would be syntactically correct.
  
Note that if $\mathfrak{p}$ is a prime ideal, then $1 \notin \mathfrak{p}$.

\begin{isabelle}
\isacommand{lemma}\ (\isakeyword{in}\ pr\_ideal)\ not\_1:\ "\isasymone \ \isasymnotin \ I"\isanewline
\isacommand{proof}\isanewline
\ \ \isacommand{assume}\ "\isasymone \ \isasymin \ I"\isanewline
\ \ \isacommand{then}\ \isacommand{have}\ "\isasymAnd x.\ \isasymlbrakk \isasymone \ \isasymin \ I;\ x\ \isasymin \ R\isasymrbrakk \ \isasymLongrightarrow \ x\ \isasymin \ I"\isanewline
\ \ \ \ \isacommand{by}\ (metis\ ideal(1)\ comm.multiplicative.right\_unit)\isanewline
\ \ \isacommand{with}\ \isacartoucheopen \isasymone \ \isasymin \ I\isacartoucheclose \ \isacommand{have}\ "I\ =\ R"\isanewline
\ \ \ \ \isacommand{by}\ auto\isanewline
\ \ \isacommand{then}\ \isacommand{show}\ False\isanewline
\ \ \ \ \isacommand{using}\ carrier\_neq\ \isacommand{by}\ blast\isanewline
\isacommand{qed}
\end{isabelle}

Let $\mathfrak{p}$ be a prime ideal and $S$ be the complement of $\mathfrak{p}$ in $R$,  we prove $S$ is a multiplicative submonoid of $R$.

\begin{isabelle}
\isacommand{lemma}\ (\isakeyword{in}\ pr\_ideal)\ submonoid\_notin:\isanewline
\ \ \isakeyword{assumes}\ "S\ =\ \{x\ \isasymin \ R.\ x\ \isasymnotin \ I\}"\isanewline
\ \ \isakeyword{shows}\ "submonoid\ S\ R\ (\isasymcdot )\ \isasymone "
\end{isabelle}

If $\mathfrak{a}$ is an ideal of $R$, $V(\mathfrak{a})$ will denote the set of all prime ideals of $R$ which contain $\mathfrak{a}$. 

\begin{isabelle}
\isacommand{definition}\ (\isakeyword{in}\ comm\_ring)\ closed\_subsets\isanewline
\ \ \ \ \ \ \ \ \ \ \ \ \ ::\ "'a\ set\ \isasymRightarrow \ ('a\ set)\ set"\ ("\isasymV \ \_"\ [900]\ 900)\isanewline
\ \ \isakeyword{where}\ "\isasymV \ \isasymaa \ \isasymequiv \ \{I.\ pr\_ideal\ R\ I\ (+)\ (\isasymcdot )\ \isasymzero \ \isasymone \ \isasymand \ \isasymaa \ \isasymsubseteq \ I\}"
\end{isabelle}

Note that $V(R) = \emptyset$ and $V(\lbrace 0 \rbrace)$ is the set of all prime ideals of $R$, abbreviated \textit{Spec} in our code.

\begin{isabelle}
\isacommand{lemma}\ (\isakeyword{in}\ comm\_ring)\ closed\_subsets\_zero:\ "\isasymV \ \{\isasymzero \}\ =\ Spec"
\end{isabelle}

\begin{isabelle}
\isacommand{lemma}\ (\isakeyword{in}\ comm\_ring)\ closed\_subsets\_R:\ "\isasymV \ R\ =\ \{\}"
\end{isabelle}

Next, we define on the set of all primes ideals of $R$ the so-called \emph{Zariski topology}\label{zariskitop} given by the subsets of the form $V(\mathfrak{a})$ as the closed subsets. We then prove in Isabelle that the set $\text{Spec}\,R$ together with its Zariski topology is a topological space.

\begin{isabelle}
\isacommand{definition}\ (\isakeyword{in}\ comm\_ring)\ is\_zariski\_open\isanewline
\ \ \ \ ::\ "'a\ set\ set\ \isasymRightarrow \ bool"\ \isakeyword{where}\isanewline
\ \ "is\_zariski\_open\ U\ \isasymequiv \ generated\_topology\ Spec\ \isanewline
\ \ \ \ \ \ \ \ \{U.\ (\isasymexists \isasymaa .\ ideal\ \isasymaa \ R\ (+)\ (\isasymcdot )\ \isasymzero \ \isasymone \ \isasymand \ U\ =\ Spec\ -\ \isasymV \ \isasymaa )\}\ U"
\end{isabelle}

\begin{isabelle}
\isacommand{lemma}\ (\isakeyword{in}\ comm\_ring)\ zariski\_is\_topological\_space:\isanewline
\ \ "topological\_space\ Spec\ is\_zariski\_open"
\end{isabelle}

In the rest of this text $\text{Spec}\,R$ will denote the set of all prime ideals of $R$ equipped with its Zariski topology.

\subsubsection*{Sheaves of Rings}

We now teach Isabelle what \emph{presheaves of rings} are. 

\begin{definition}[presheaf of rings]
	Let $X$ be a topological space. A presheaf $\mathscr{F}$ of rings on $X$ consists of the following data:
	\begin{itemize}
		\item for every open set $U$, a ring $\mathscr{F}(U)$
		\item for every inclusion $V \subseteq U$ of open subsets, a morphism of rings $\rho_{UV}: \mathscr{F}(U) \rightarrow \mathscr{F}(V)$  
	\end{itemize}
satisfying 
	\begin{enumerate}
		\item $\mathscr{F}(\emptyset) = \lbrace 0 \rbrace$
		\item $\rho_{UU}$ is the identity map for every open subset $U$
		\item  If $W \subseteq V \subseteq U$ are three open subsets, then $\rho_{UW} = \rho_{VW} \circ \rho_{UV}$.
	\end{enumerate}
\end{definition}

\begin{isabelle}
\isacommand{locale}\ presheaf\_of\_rings\ =\ topological\_space\isanewline
\ \ +\ \isakeyword{fixes}\ \isasymFF ::\ "'a\ set\ \isasymRightarrow \ 'b\ set"\isanewline
\ \ \isakeyword{and}\ \isasymrho ::\ "'a\ set\ \isasymRightarrow \ 'a\ set\ \isasymRightarrow \ ('b\ \isasymRightarrow \ 'b)"\ \isakeyword{and}\ b::\ "'b"\isanewline
\ \ \isakeyword{and}\ add\_str::\ "'a\ set\ \isasymRightarrow \ ('b\ \isasymRightarrow \ 'b\ \isasymRightarrow \ 'b)"\ ("+\isactrlbsub \_\isactrlesub ")\isanewline
\ \ \isakeyword{and}\ mult\_str::\ "'a\ set\ \isasymRightarrow \ ('b\ \isasymRightarrow \ 'b\ \isasymRightarrow \ 'b)"\ ("\isasymcdot \isactrlbsub \_\isactrlesub ")\isanewline
\ \ \isakeyword{and}\ zero\_str::\ "'a\ set\ \isasymRightarrow \ 'b"\ ("\isasymzero \isactrlbsub \_\isactrlesub ")\ \isanewline
\ \ \isakeyword{and}\ one\_str::\ "'a\ set\ \isasymRightarrow \ 'b"\ ("\isasymone \isactrlbsub \_\isactrlesub ")\isanewline
\ \ \isakeyword{assumes}\ is\_ring\_morphism:\isanewline
\ \ \ \ "\isasymAnd U\ V.\ is\_open\ U\ \isasymLongrightarrow \ is\_open\ V\ \isasymLongrightarrow \ V\ \isasymsubseteq \ U\ \isanewline
\ \ \ \ \ \ \ \ \ \ \ \ \ \ \ \ \isasymLongrightarrow \ ring\_homomorphism\ (\isasymrho \ U\ V)\isanewline
\ \ \ \ \ \ \ \ \ \ \ \ \ \ \ \ \ \ \ \ \ \ \ \ \ \ (\isasymFF \ U)\ (+\isactrlbsub U\isactrlesub )\ (\isasymcdot \isactrlbsub U\isactrlesub )\ \isasymzero \isactrlbsub U\isactrlesub \ \isasymone \isactrlbsub U\isactrlesub \isanewline
\ \ \ \ \ \ \ \ \ \ \ \ \ \ \ \ \ \ \ \ \ \ \ \ \ \ (\isasymFF \ V)\ (+\isactrlbsub V\isactrlesub )\ (\isasymcdot \isactrlbsub V\isactrlesub )\ \isasymzero \isactrlbsub V\isactrlesub \ \isasymone \isactrlbsub V\isactrlesub "\isanewline
\ \ \isakeyword{and}\ ring\_of\_empty:\ "\isasymFF \ \{\}\ =\ \{b\}"\isanewline
\ \ \isakeyword{and}\ identity\_map\ [simp]:\ "\isasymAnd U.\ is\_open\ U\ \isanewline
\ \ \ \ \ \ \ \ \ \ \ \ \ \ \ \ \ \ \ \ \ \ \isasymLongrightarrow \ (\isasymAnd x.\ x\ \isasymin \ \isasymFF \ U\ \isasymLongrightarrow \ \isasymrho \ U\ U\ x\ =\ x)"\isanewline
\ \ \isakeyword{and}\ assoc\_comp:\ \isanewline
\ \ \ \ "\isasymAnd U\ V\ W\ x.\ \isasymlbrakk is\_open\ U;\ is\_open\ V;\ is\_open\ W;V\ \isasymsubseteq \ U;\ \isanewline
\ \ \ \ \ \ \ \ \ W\ \isasymsubseteq \ V;x\ \isasymin \ (\isasymFF \ U)\isasymrbrakk \ \isasymLongrightarrow \ \isasymrho \ U\ W\ x\ =\ (\isasymrho \ V\ W\ \isasymcirc \ \isasymrho \ U\ V)\ x"
\end{isabelle}

Within the locale the sets $\mathscr{F}(U)$'s are endowed with ring structures using the functional \textit{add\_str} (\textit{resp.} \textit{mult\_str}, \textit{zero\_str}, \textit{one\_str}) that takes an open subset $U$ and outputs the addition (\textit{resp.} the multiplication, the additive unit, the multiplicative unit) on the set $\mathscr{F}(U)$.
Ring homomorphisms have been defined in \cite{ballarin-exploring} as follows. 

\begin{isabelle}
	\isacommand{locale}\isamarkupfalse%
	\ ring{\isacharunderscore}{\kern0pt}homomorphism\ {\isacharequal}{\kern0pt}\isanewline
	\ \ map\ {\isasymeta}\ R\ R{\isacharprime}{\kern0pt}\ {\isacharplus}{\kern0pt}\ source{\isacharcolon}{\kern0pt}\ ring\ R\ {\isachardoublequoteopen}{\isacharparenleft}{\kern0pt}{\isacharplus}{\kern0pt}{\isacharparenright}{\kern0pt}{\isachardoublequoteclose}\ {\isachardoublequoteopen}{\isacharparenleft}{\kern0pt}{\isasymcdot}{\isacharparenright}{\kern0pt}{\isachardoublequoteclose}\ {\isasymzero}\ {\isasymone}\ \isanewline
	\ \ {\isacharplus}{\kern0pt}\ target{\isacharcolon}{\kern0pt}\ ring\ R{\isacharprime}{\kern0pt}\ {\isachardoublequoteopen}{\isacharparenleft}{\kern0pt}{\isacharplus}{\kern0pt}{\isacharprime}{\kern0pt}{\isacharparenright}{\kern0pt}{\isachardoublequoteclose}\ {\isachardoublequoteopen}{\isacharparenleft}{\kern0pt}{\isasymcdot}{\isacharprime}{\kern0pt}{\isacharparenright}{\kern0pt}{\isachardoublequoteclose}\ {\isachardoublequoteopen}{\isasymzero}{\isacharprime}{\kern0pt}{\isachardoublequoteclose}\ {\isachardoublequoteopen}{\isasymone}{\isacharprime}{\kern0pt}{\isachardoublequoteclose}\ \isanewline
	\ \ {\isacharplus}{\kern0pt}\ additive{\isacharcolon}{\kern0pt}\ group{\isacharunderscore}{\kern0pt}homomorphism\ {\isasymeta}\ R\ {\isachardoublequoteopen}{\isacharparenleft}{\kern0pt}{\isacharplus}{\kern0pt}{\isacharparenright}{\kern0pt}{\isachardoublequoteclose}\ {\isasymzero}\ R{\isacharprime}{\kern0pt}\ {\isachardoublequoteopen}{\isacharparenleft}{\kern0pt}{\isacharplus}{\kern0pt}{\isacharprime}{\kern0pt}{\isacharparenright}{\kern0pt}{\isachardoublequoteclose}\ {\isachardoublequoteopen}{\isasymzero}{\isacharprime}{\kern0pt}{\isachardoublequoteclose}\ \isanewline
	\ \ {\isacharplus}{\kern0pt}\ multiplicative{\isacharcolon}{\kern0pt}\ monoid{\isacharunderscore}{\kern0pt}homomorphism\ {\isasymeta}\ R\ {\isachardoublequoteopen}{\isacharparenleft}{\kern0pt}{\isasymcdot}{\isacharparenright}{\kern0pt}{\isachardoublequoteclose}\ {\isasymone}\ R{\isacharprime}{\kern0pt}\ {\isachardoublequoteopen}{\isacharparenleft}{\kern0pt}{\isasymcdot}{\isacharprime}{\kern0pt}{\isacharparenright}{\kern0pt}{\isachardoublequoteclose}\ {\isachardoublequoteopen}{\isasymone}{\isacharprime}{\kern0pt}{\isachardoublequoteclose}\isanewline
	\ \ \isakeyword{for}\ {\isasymeta}\ \isakeyword{and}\ R\ \isakeyword{and}\ addition\ {\isacharparenleft}{\kern0pt}\isakeyword{infixl}\ {\isachardoublequoteopen}{\isacharplus}{\kern0pt}{\isachardoublequoteclose}\ {\isadigit{6}}{\isadigit{5}}{\isacharparenright}{\kern0pt}\ \isakeyword{and}\ multiplication\ {\isacharparenleft}{\kern0pt}\isakeyword{infixl}\ {\isachardoublequoteopen}{\isasymcdot}{\isachardoublequoteclose}\ {\isadigit{7}}{\isadigit{0}}{\isacharparenright}{\kern0pt}\ \isanewline
	\ \ \ \ \isakeyword{and}\ zero\ {\isacharparenleft}{\kern0pt}{\isachardoublequoteopen}{\isasymzero}{\isachardoublequoteclose}{\isacharparenright}{\kern0pt}\ \isakeyword{and}\ unit\ {\isacharparenleft}{\kern0pt}{\isachardoublequoteopen}{\isasymone}{\isachardoublequoteclose}{\isacharparenright}{\kern0pt}\ \isakeyword{and}\ R{\isacharprime}{\kern0pt}\ \isakeyword{and}\ addition{\isacharprime}{\kern0pt}\ {\isacharparenleft}{\kern0pt}\isakeyword{infixl}\ {\isachardoublequoteopen}{\isacharplus}{\kern0pt}{\isacharprime}{\kern0pt}{\isacharprime}{\kern0pt}{\isachardoublequoteclose}\ {\isadigit{6}}{\isadigit{5}}{\isacharparenright}{\kern0pt}\ \isanewline
	\ \ \ \ \isakeyword{and}\ multiplication{\isacharprime}{\kern0pt}\ {\isacharparenleft}{\kern0pt}\isakeyword{infixl}\ {\isachardoublequoteopen}{\isasymcdot}{\isacharprime}{\kern0pt}{\isacharprime}{\kern0pt}{\isachardoublequoteclose}\ {\isadigit{7}}{\isadigit{0}}{\isacharparenright}{\kern0pt}\ \isakeyword{and}\ zero{\isacharprime}{\kern0pt}\ {\isacharparenleft}{\kern0pt}{\isachardoublequoteopen}{\isasymzero}{\isacharprime}{\kern0pt}{\isacharprime}{\kern0pt}{\isachardoublequoteclose}{\isacharparenright}{\kern0pt}\ \isakeyword{and}\ unit{\isacharprime}{\kern0pt}\ {\isacharparenleft}{\kern0pt}{\isachardoublequoteopen}{\isasymone}{\isacharprime}{\kern0pt}{\isacharprime}{\kern0pt}{\isachardoublequoteclose}{\isacharparenright}{\kern0pt}
\end{isabelle}

As expected the source and target of a ring homomorphism are rings.
As a consequence, the condition \textit{is\_ring\_morphism} within the locale \textit{presheaf\_of\_rings}, requiring $\rho_{UV}$ to be a ring homomorphism, implies that $\mathscr{F}(U)$ together with its structure is a ring for every open subset $U$ of the underlying topological space.

\begin{isabelle}
\isacommand{lemma}\ (\isakeyword{in}\ presheaf\_of\_rings)\ is\_ring\_from\_is\_homomorphism:\isanewline
\ \ \isakeyword{fixes}\ U::\ "'a\ set"\isanewline
\ \ \isakeyword{assumes}\ "is\_open\ U"\isanewline
\ \ \isakeyword{shows}\ "ring\ (\isasymFF \ U)\ (+\isactrlbsub U\isactrlesub )\ (\isasymcdot \isactrlbsub U\isactrlesub )\ \isasymzero \isactrlbsub U\isactrlesub \ \isasymone \isactrlbsub U\isactrlesub "
\end{isabelle}

\begin{notation}
	The elements of $\mathscr{F}(U)$ are sometimes called the sections of the presheaf $\mathscr{F}$ over $U$ and given $s \in \mathscr{F}(U)$, $s\restriction V$ denotes the element $\rho_{UV}(s)$.
\end{notation}

Of course, we have the corresponding notion of morphisms.

\begin{definition}[morphism of presheaves of rings]
	A morphism $\phi: \mathscr{F} \rightarrow \mathscr{F}'$ of presheaves of rings on a topological space $X$ is given by a morphism $\phi_U: \mathscr{F}(U) \rightarrow \mathscr{F}'(U)$ for each open subset $U$ of $X$ such that  the following diagram commutes
	\[
	\begin{tikzcd}
	\mathscr{F}(U) \arrow[r, "\phi_U"] \arrow[d, "\rho_{UV}"] & \mathscr{F}'(U) \arrow[d, "\rho'_{UV}"] \\
	\mathscr{F}(V) \arrow[r, "\phi_V"] & \mathscr{F}'(V) 
	\end{tikzcd}
	\]
	for every inclusion $V \subseteq U$.
\end{definition}

\begin{isabelle}
\isacommand{locale}\ morphism\_presheaves\_of\_rings\ =\isanewline
\ \ \ \ source:\ presheaf\_of\_rings\ X\ is\_open\ \isasymFF \ \isasymrho \ b\ add\_str\isanewline
\ \ \ \ \ \ \ \ \ \ \ \ \ \ \ \ \ \ \ \ \ \ \ \ \ \ \ \ \ \ \ \ mult\_str\ zero\_str\ one\_str\isanewline
\ \ +\ target:\ presheaf\_of\_rings\ X\ is\_open\ \isasymFF '\ \isasymrho '\ b'\ add\_str'\ \isanewline
\ \ \ \ \ \ \ \ \ \ \ \ \ \ \ \ \ \ \ \ \ \ \ \ \ \ \ \ \ mult\_str'\ zero\_str'\ one\_str'\isanewline
\ \ \isakeyword{for}\ X\ \isakeyword{and}\ is\_open\isanewline
\ \ \ \ \isakeyword{and}\ \isasymFF \ \isakeyword{and}\ \isasymrho \ \isakeyword{and}\ b\ \isakeyword{and}\ add\_str\ ("+\isactrlbsub \_\isactrlesub ")\ \isanewline
\ \ \ \ \isakeyword{and}\ mult\_str\ ("\isasymcdot \isactrlbsub \_\isactrlesub ")\ \isakeyword{and}\ zero\_str\ ("\isasymzero \isactrlbsub \_\isactrlesub ")\ \isanewline
\ \ \ \ \isakeyword{and}\ one\_str\ ("\isasymone \isactrlbsub \_\isactrlesub ")\ \isakeyword{and}\ \isasymFF '\ \isakeyword{and}\ \isasymrho '\ \isakeyword{and}\ b'\ \isanewline
\ \ \ \ \isakeyword{and}\ add\_str'\ ("+'{\kern0pt}'\isactrlbsub \_\isactrlesub ")\ \isakeyword{and}\ mult\_str'\ ("\isasymcdot '{\kern0pt}'\isactrlbsub \_\isactrlesub ")\isanewline
\ \ \ \ \isakeyword{and}\ zero\_str'\ ("\isasymzero '{\kern0pt}'\isactrlbsub \_\isactrlesub ")\ \isakeyword{and}\ one\_str'\ ("\isasymone '{\kern0pt}'\isactrlbsub \_\isactrlesub ")\ +\isanewline
\ \ \isakeyword{fixes}\ fam\_morphisms::\ "'a\ set\ \isasymRightarrow \ ('b\ \isasymRightarrow \ 'c)"\isanewline
\ \ \isakeyword{assumes}\ is\_ring\_morphism:\ \isanewline
\ \ \ \ "\isasymAnd U.\ is\_open\ U\ \isasymLongrightarrow \ ring\_homomorphism\ (fam\_morphisms\ U)\isanewline
\ \ \ \ \ \ \ \ \ \ \ \ \ \ \ \ \ \ \ \ \ \ \ \ \ \ \ \ \ \ (\isasymFF \ U)\ (+\isactrlbsub U\isactrlesub )\ (\isasymcdot \isactrlbsub U\isactrlesub )\ \isasymzero \isactrlbsub U\isactrlesub \ \isasymone \isactrlbsub U\isactrlesub \isanewline
\ \ \ \ \ \ \ \ \ \ \ \ \ \ \ \ \ \ \ \ \ \ \ \ \ (\isasymFF '\ U)\ (+'\isactrlbsub U\isactrlesub )\ (\isasymcdot '\isactrlbsub U\isactrlesub )\ \isasymzero '\isactrlbsub U\isactrlesub \ \isasymone '\isactrlbsub U\isactrlesub "\isanewline
\ \ \ \ \isakeyword{and}\ comm\_diagrams:\ \isanewline
\ \ \ \ "\isasymAnd U\ V\ x.\ \isasymlbrakk is\_open\ U;\ is\_open\ V;\ V\ \isasymsubseteq \ U;x\ \isasymin \ \isasymFF \ U\ \isasymrbrakk \isanewline
\ \ \ \ \ \ \ \ \ \ \ \ \ \ \isasymLongrightarrow \ (\isasymrho '\ U\ V\ \isasymcirc \ fam\_morphisms\ U)\ x\ \isanewline
\ \ \ \ \ \ \ \ \ \ \ \ \ \ \ \ \ \ \ \ \ \ \ \ \ \ \ \ \ \ =\ (fam\_morphisms\ V\ \isasymcirc \ \isasymrho \ U\ V)\ x"
\end{isabelle}

In the snippet of code above, we see the inheritance of locales in action. Indeed, the locale \textit{presheaf\_of\_rings} inherits multiple instances, one for the source and one for the target of the notion of morphism being defined.
The notion of composition for morphisms of presheaves is straightforward and we check in Isabelle without any difficulty that the composition is again a morphism between presheaves of rings. \\
Indeed,	let $\phi: \mathscr{F} \rightarrow \mathscr{G}$ (\textit{resp.} $\psi: \mathscr{G} \rightarrow \mathscr{H}$) be a morphism of presheaves of rings on a topological space $X$. One defines the composition $\psi \circ \phi$ from $\mathscr{F}$ to $\mathscr{H}$ as 
	\[
	(\psi \circ \phi)_U \coloneqq \psi_U \circ \phi_U
	\]
	for every open subset $U$ of $X$.

\begin{isabelle}
\isacommand{lemma}\ comp\_of\_presheaves:\isanewline
\ \ \isakeyword{assumes}\ "morphism\_presheaves\_of\_rings\ X\ is\_open\ \isasymFF \ \isasymrho \ b\ \isanewline
\ \ \ \ \ \ \ \ \ \ \ \ add\_str\ mult\_str\ zero\_str\ one\_str\ \isasymFF '\ \isasymrho '\ b'\ \isanewline
\ \ \ \ \ \ \ \ \ \ \ \ \ \ \ \ add\_str'\ mult\_str'\ zero\_str'\ one\_str'\ \isasymphi "\isanewline
\ \ \ \ \isakeyword{and}\ "morphism\_presheaves\_of\_rings\ X\ is\_open\ \isasymFF '\ \isasymrho '\ b'\isanewline
\ \ \ \ \ \ \ \ \ \ add\_str'\ mult\_str'\ zero\_str'\ one\_str'\ \isasymFF '{\kern0pt}'\ \isasymrho '{\kern0pt}'\ \isanewline
\ \ \ \ \ \ \ \ \ \ b'{\kern0pt}'\ add\_str'{\kern0pt}'\ mult\_str'{\kern0pt}'\ zero\_str'{\kern0pt}'\ one\_str'{\kern0pt}'\ \isasymphi '"\isanewline
\ \ \isakeyword{shows}\ "morphism\_presheaves\_of\_rings\ X\ is\_open\ \isasymFF \ \isasymrho \ b\ add\_str\isanewline
\ \ \ \ \ \ mult\_str\ zero\_str\ one\_str\ \isasymFF '{\kern0pt}'\ \isasymrho '{\kern0pt}'\ b'{\kern0pt}'\ add\_str'{\kern0pt}'\ \isanewline
\ \ \ \ \ \ mult\_str'{\kern0pt}'\ zero\_str'{\kern0pt}'\ one\_str'{\kern0pt}'\ (\isasymlambda U.\ (\isasymphi '\ U\ \isasymcirc \ \isasymphi \ U\ \isasymdown \ \isasymFF \ U))"
\end{isabelle}

The syntax $g \circ f \downarrow D$, used in the last line of the above code, is simply some syntactic sugar for the definition of the composition $g \circ f$ of two maps $f$ and $g$ on the domain $D$.  

\begin{remark}
	Given a presheaf $\mathscr{F}$ on $X$, let us define $(1_{\mathscr{F}})_U$ as the identity morphism of $\mathscr{F}(U)$ for every open subset $U$ of $X$. The family $1_{\mathscr{F}}$ is obviously a morphism of presheaves from $\mathscr{F}$ to itself. A morphism $\phi: \mathscr{F} \rightarrow \mathscr{G}$ of presheaves of rings is an isomorphism if and only if there exists a morphism $\psi: \mathscr{G} \rightarrow \mathscr{F}$ such that $\psi \circ \phi = 1_{\mathscr{F}}$ and $\phi \circ \psi = 1_{\mathscr{G}}$.  
\end{remark}

\begin{isabelle}
\isacommand{locale}\ iso\_presheaves\_of\_rings\ \isanewline
\ \ =\ mor:morphism\_presheaves\_of\_rings\ +\ \isanewline
\ \ \isakeyword{assumes}\ is\_inv:\ "\isasymexists \isasympsi .\ morphism\_presheaves\_of\_rings\ X\ is\_open\ \isanewline
\ \ \ \ \isasymFF '\ \isasymrho '\ b'\ add\_str'\ mult\_str'\ zero\_str'\ one\_str'\ \isasymFF \ \isasymrho \ b\ \isanewline
\ \ \ \ add\_str\ mult\_str\ zero\_str\ one\_str\ \isasympsi \ \isasymand \ (\isasymforall U.\ is\_open\ U\ \isanewline
\ \ \ \ \ \ \ \ \ \ \isasymlongrightarrow \ (\isasymforall x\ \isasymin \ (\isasymFF '\ U).\ (fam\_morphisms\ U\ \isasymcirc \ \isasympsi \ U)\ x\ =\ x)\ \isanewline
\ \ \ \ \ \ \ \ \ \ \ \ \ \ \ \ \isasymand \ (\isasymforall x\ \isasymin \ (\isasymFF \ U).\ (\isasympsi \ U\ \isasymcirc \ fam\_morphisms\ U)\ x\ =\ x))"
\end{isabelle}

Now, we introduce the notion of a \emph{sheaf of rings}. 			

\begin{definition}[sheaf of rings]
	A sheaf of rings on a topological space $X$ is a presheaf of rings on $X$ that satisfies the following additional properties:
	\begin{enumerate}
		\item[(locality)] Given an open set $U$, $\lbrace V_i \rbrace$ an open covering of $U$ and $s \in \mathscr{F}(U)$ such that $s \restriction V_i = 0$ for all $i$, then one has $s = 0$.
		\item[(glueing)] Given an open set $U$, $\lbrace V_i \rbrace$ an open covering of $U$ and elements $s_i \in \mathscr{F}(V_i)$ satisfying the property $s_i \restriction V_i \cap V_j = s_j \restriction V_i \cap V_j$ for all $i$ and $j$, then there exists an element $s \in \mathscr{F}(U)$ such that $s \restriction V_i = s_i$ for all $i$.  
	\end{enumerate}	
\end{definition}

The resulting formalization of sheaves of rings in Isabelle is concise (10 lines), tidy and almost transparent with respect to the pen-and-paper definition.

\begin{isabelle}
\isacommand{locale}\ sheaf\_of\_rings\ =\ presheaf\_of\_rings\ +\isanewline
\ \ \isakeyword{assumes}\ locality:\ "\isasymAnd U\ I\ V\ s.\ open\_cover\_of\_open\_subset\ S\ \isanewline
\ \ \ \ \ \ \ \ \ \ is\_open\ U\ I\ V\ \isasymLongrightarrow \ (\isasymAnd i.\ i\isasymin I\ \isasymLongrightarrow \ V\ i\ \isasymsubseteq \ U)\ \isasymLongrightarrow \ s\ \isasymin \ \isasymFF \ U\isanewline
\ \ \ \ \ \ \ \ \ \ \ \isasymLongrightarrow \ (\isasymAnd i.\ i\isasymin I\ \isasymLongrightarrow \ \isasymrho \ U\ (V\ i)\ s\ =\ \isasymzero \isactrlbsub (V\ i)\isactrlesub )\ \isasymLongrightarrow \ s\ =\ \isasymzero \isactrlbsub U\isactrlesub "\isanewline
\ \ \ \ \isakeyword{and}\ glueing:\ \ \isanewline
\ \ \ \ \ \ "\isasymAnd U\ I\ V\ s.\ open\_cover\_of\_open\_subset\ S\ is\_open\ U\ I\ V\ \isanewline
\ \ \ \ \ \ \ \ \isasymLongrightarrow \ (\isasymforall i.\ i\isasymin I\ \isasymlongrightarrow \ V\ i\ \isasymsubseteq \ U\ \isasymand \ s\ i\ \isasymin \ \isasymFF \ (V\ i))\ \isanewline
\ \ \ \ \ \ \ \ \isasymLongrightarrow \ (\isasymAnd i\ j.\ i\isasymin I\ \isasymLongrightarrow \ j\isasymin I\ \isasymLongrightarrow \ \isasymrho \ (V\ i)\ (V\ i\ \isasyminter \ V\ j)\ (s\ i)\isanewline
\ \ \ \ \ \ \ \ \ \ \ \ \ \ \ \ \ \ =\ \isasymrho \ (V\ j)\ (V\ i\ \isasyminter \ V\ j)\ (s\ j))\ \isanewline
\ \ \ \ \ \ \ \ \isasymLongrightarrow \ (\isasymexists t.\ t\ \isasymin \ \isasymFF \ U\ \isasymand \ (\isasymforall i.\ i\isasymin I\ \isasymlongrightarrow \ \isasymrho \ U\ (V\ i)\ t\ =\ s\ i))"
\end{isabelle}

A morphism (\textit{resp.} an isomorphism) of sheaves of rings is nothing but a morphism (\textit{resp.} an isomorphism) of presheaves of rings, hence the following locale in Isabelle.

\begin{isabelle}
\isacommand{locale}\ morphism\_sheaves\_of\_rings\ =\ morphism\_presheaves\_of\_rings
\end{isabelle}

Let $X$ be a topological space, $\mathscr{F}$ a sheaf of rings on $X$ and $U$ an open subset of $X$,  one proves that the restriction of $\mathscr{F}$ to $U$, seen as a topological subspace with respect to the induced topology, is a sheaf of rings. In our formal development this sheaf is called the induced sheaf and shorten \textit{ind\_sheaf}.
We encapsulate the relevant mathematical context (the ``let be'' part) in a dedicated locale and formalize within this locale the relevant definitions and we eventually prove the \emph{induced sheaf $\mathscr{F}|_U$} is indeed a sheaf of rings.
		
\begin{isabelle}
\isacommand{locale}\ ind\_sheaf\ =\ sheaf\_of\_rings\ +\isanewline
\ \ \isakeyword{fixes}\ U::\ "'a\ set"\isanewline
\ \ \isakeyword{assumes}\ is\_open\_subset:\ "is\_open\ U"\isanewline
\isakeyword{begin}\isanewline
\isanewline
\isacommand{interpretation}\ it:\ ind\_topology\ S\ is\_open\ U\isanewline

\isacommand{definition}\ ind\_sheaf::\ "'a\ set\ \isasymRightarrow \ 'b\ set"\isanewline
\ \ \isakeyword{where}\ "ind\_sheaf\ V\ \isasymequiv \ \isasymFF \ (U\ \isasyminter \ V)"\isanewline
\isanewline
\isacommand{definition}\ ind\_ring\_morphisms::\ "'a\ set\ \isasymRightarrow \ 'a\ set\ \isasymRightarrow \ ('b\ \isasymRightarrow \ 'b)"\isanewline
\ \ \isakeyword{where}\ "ind\_ring\_morphisms\ V\ W\ \isasymequiv \ \isasymrho \ (U\ \isasyminter \ V)\ (U\ \isasyminter \ W)"\isanewline
\isanewline
\isacommand{definition}\ ind\_add\_str::\ "'a\ set\ \isasymRightarrow \ ('b\ \isasymRightarrow \ 'b\ \isasymRightarrow \ 'b)"\isanewline
\ \ \isakeyword{where}\ "ind\_add\_str\ V\ \isasymequiv \ \isasymlambda x\ y.\ +\isactrlbsub (U\ \isasyminter \ V)\isactrlesub \ x\ y"\isanewline
\isanewline
\isacommand{definition}\ ind\_mult\_str::\ "'a\ set\ \isasymRightarrow \ ('b\ \isasymRightarrow \ 'b\ \isasymRightarrow \ 'b)"\isanewline
\ \ \isakeyword{where}\ "ind\_mult\_str\ V\ \isasymequiv \ \isasymlambda x\ y.\ \isasymcdot \isactrlbsub (U\ \isasyminter \ V)\isactrlesub \ x\ y"\isanewline
\isanewline
\isacommand{definition}\ ind\_zero\_str::\ "'a\ set\ \isasymRightarrow \ 'b"\isanewline
\ \ \isakeyword{where}\ "ind\_zero\_str\ V\ \isasymequiv \ \isasymzero \isactrlbsub (U\isasyminter V)\isactrlesub "\isanewline
\isanewline
\isacommand{definition}\ ind\_one\_str::\ "'a\ set\ \isasymRightarrow \ 'b"\isanewline
\ \ \isakeyword{where}\ "ind\_one\_str\ V\ \isasymequiv \ \isasymone \isactrlbsub (U\isasyminter V)\isactrlesub "\isanewline
\isanewline
\isacommand{lemma}\ ind\_sheaf\_is\_sheaf:\isanewline
\ \ "sheaf\_of\_rings\ U\ it.ind\_is\_open\ ind\_sheaf\ ind\_ring\_morphisms\ \isanewline
\ \ \ \ \ \ \ \ \ \ \ \ b\ ind\_add\_str\ ind\_mult\_str\ ind\_zero\_str\ ind\_one\_str"\isanewline

\isacommand{end}
\end{isabelle}
In the code above the \textbf{interpretation} mechanism for locales makes possible to instantiate the locale \textit{ind\_topology}, for the induced topology on a subset, with a list of arguments. Then it becomes possible to refer later to the induced topology on the open subset $U$ with the abbreviation \textit{it}.

We introduce a second construction that takes as input a given sheaf of rings and outputs a new sheaf of rings. Let $f: X \rightarrow Y$ be a continuous map between topological spaces and $\mathscr{F}$ a sheaf of rings on $X$. Given an open subset $V$ of $Y$, define $(f_{*} \mathscr{F})(V)$ to be $\mathscr{F} (f^{-1}(V))$. We then prove that $f_{*} \mathscr{F}$, together with the obvious structure, is a sheaf of rings on the topological space $Y$. The sheaf $f_{*} \mathscr{F}$ is called the \emph{direct image of  $\mathscr{F}$}.  

\begin{isabelle}
\isacommand{locale}\ im\_sheaf\ =\ sheaf\_of\_rings\ +\ continuous\_map\isanewline
\isakeyword{begin}\isanewline
\isanewline
\isacommand{definition}\ im\_sheaf::\ "'c\ set\ =>\ 'b\ set"\isanewline
\ \ \isakeyword{where}\ "im\_sheaf\ V\ \isasymequiv \ \isasymFF \ (f\isactrlsup \isasyminverse \ S\ V)"\isanewline
\isanewline
\isacommand{definition}\ im\_sheaf\_morphisms::\ "'c\ set\ \isasymRightarrow \ 'c\ set\ \isasymRightarrow \ ('b\ \isasymRightarrow \ 'b)"\isanewline
\ \ \isakeyword{where}\ "im\_sheaf\_morphisms\ U\ V\ \isasymequiv \ \isasymrho \ (f\isactrlsup \isasyminverse \ S\ U)\ (f\isactrlsup \isasyminverse \ S\ V)"\isanewline
\isanewline
\isacommand{definition}\ add\_im\_sheaf::\ "'c\ set\ \isasymRightarrow \ 'b\ \isasymRightarrow \ 'b\ \isasymRightarrow \ 'b"\isanewline
\ \ \isakeyword{where}\ "add\_im\_sheaf\ \isasymequiv \ \isasymlambda V\ x\ y.\ +\isactrlbsub (f\isactrlsup \isasyminverse \ S\ V)\isactrlesub \ x\ y"\isanewline
\isanewline
\isacommand{definition}\ mult\_im\_sheaf::\ "'c\ set\ \isasymRightarrow \ 'b\ \isasymRightarrow \ 'b\ \isasymRightarrow \ 'b"\isanewline
\ \ \isakeyword{where}\ "mult\_im\_sheaf\ \isasymequiv \ \isasymlambda V\ x\ y.\ \isasymcdot \isactrlbsub (f\isactrlsup \isasyminverse \ S\ V)\isactrlesub \ x\ y"\isanewline
\isanewline
\isacommand{definition}\ zero\_im\_sheaf::\ "'c\ set\ \isasymRightarrow \ 'b"\isanewline
\ \ \isakeyword{where}\ "zero\_im\_sheaf\ \isasymequiv \ \isasymlambda V.\ \isasymzero \isactrlbsub (f\isactrlsup \isasyminverse \ S\ V)\isactrlesub "\isanewline
\isanewline
\isacommand{definition}\ one\_im\_sheaf::\ "'c\ set\ \isasymRightarrow \ 'b"\isanewline
\ \ \isakeyword{where}\ "one\_im\_sheaf\ \isasymequiv \ \isasymlambda V.\ \isasymone \isactrlbsub (f\isactrlsup \isasyminverse \ S\ V)\isactrlesub "\isanewline
\isanewline
\isacommand{lemma}\ im\_sheaf\_is\_sheaf:\isanewline
\ \ "sheaf\_of\_rings\ S'\ is\_open'\ im\_sheaf\ im\_sheaf\_morphisms\ b\isanewline
\ \ \ \ \ \ \ \ \ \ add\_im\_sheaf\ mult\_im\_sheaf\ zero\_im\_sheaf\ one\_im\_sheaf"\isanewline
\isacommand{end}
\end{isabelle}

As can been seen in the snippet of code above, we encapsulate again the context of the construction in a dedicated locale and this lets us formalize the relevant definitions in this context and we eventually prove the direct image of a sheaf is again a sheaf. 

\subsubsection*{Localizations of Rings}

Next, we introduce the localization of commutative rings.
Let $S$ be a multiplicative submonoid of $R$. We consider pairs $(r, s)$ with $r \in R$ and $s \in S$ and we define the following relation on them
	\[
	(r, s) \sim (r', s')
	\]
if and only if there exists $s_1 \in S$ such that $s_1(s' r - s r') = 0$. One checks that the relation $\sim$ is an equivalence relation. The equivalence class of a pair $(r, s)$ is denoted by $r/s$ and the set of equivalence classes is denoted by $S^{-1} R$.
We define the following addition on $S^{-1} R$.
	\[
	\frac{r}{s} + \frac{r'}{s'} = \frac{r s' + r' s}{s s'}
	\]
We also define the following multiplication on $S^{-1} R$. 
	\[
	\frac{r}{s} \times \frac{r'}{s'} = \frac{r r'}{s s'}
	\]
One checks these operations are well-defined with $(S^{-1} R, 0/1, +, 1/1, \times)$ forming a ring. Following Serge Lang \cite[II.4]{lang2002algebra} we call the new ring $S^{-1} R$ the \emph{quotient ring of $R$ by $S$}\footnote{It is also called the \textit{ring of fractions of $R$ by $S$} or the \textit{localization of $R$ at $S$}. No confusion with the \textit{factor ring} should be possible, since $S$ is not an ideal of $R$ but a multiplicative submonoid.}.
This construction translates into the following locale.

\begin{isabelle}
\isacommand{locale}\ quotient\_ring\ =\ comm:comm\_ring\ R\ "(+)"\ "(\isasymcdot )"\ "\isasymzero "\ "\isasymone "\ \isanewline
\ \ \ \ \ \ \ \ \ \ +\ submonoid\ S\ R\ "(\isasymcdot )"\ "\isasymone "\isanewline
\ \ \ \ \ \ \isakeyword{for}\ S\ \isakeyword{and}\ R\ \isakeyword{and}\ addition\ (\isakeyword{infixl}\ "+"\ 65)\ \isanewline
\ \ \ \ \ \ \ \ \isakeyword{and}\ multiplication\ (\isakeyword{infixl}\ "\isasymcdot "\ 70)\ \isakeyword{and}\ zero\ ("\isasymzero ")\ \isakeyword{and}\ unit\ ("\isasymone ")\isanewline
\isakeyword{begin}\isanewline
\isanewline
\isacommand{definition}\ rel::\ "('a\ \isasymtimes \ 'a)\ \isasymRightarrow \ ('a\ \isasymtimes \ 'a)\ \isasymRightarrow \ bool"\ (\isakeyword{infix}\ "\isasymsim "\ 80)\isanewline
\ \ \isakeyword{where}\ "x\ \isasymsim \ y\ \isasymequiv \ \isasymexists s1.\ s1\ \isasymin \ S\ \isasymand \ s1\ \isasymcdot \ (snd\ y\ \isasymcdot \ fst\ x\ -\ snd\ x\ \isasymcdot \ fst\ y)\ =\ \isasymzero "\isanewline
\isanewline
\isacommand{interpretation}\ rel:\ equivalence\ "R\ \isasymtimes \ S"\ "\{(x,y)\ \isasymin \ (R\isasymtimes S)\isasymtimes (R\isasymtimes S).\ x\ \isasymsim \ y\}"\isanewline
\isanewline
\isacommand{definition}\ frac::\ "'a\ \isasymRightarrow \ 'a\ \isasymRightarrow \ ('a\ \isasymtimes \ 'a)\ set"\ (\isakeyword{infixl}\ "'/"\ 75)\isanewline
\ \ \isakeyword{where}\ "r\ /\ s\ \isasymequiv \ rel.Class\ (r,\ s)"\isanewline
\isanewline
\isacommand{lemma}\ add\_rel\_frac:\isanewline
\ \ \isakeyword{assumes}\ "(r,s)\ \isasymin \ R\ \isasymtimes \ S"\ \isakeyword{and}\ "(r',s')\isasymin \ R\ \isasymtimes \ S"\isanewline
\ \ \isakeyword{shows}\ "add\_rel\ (r/s)\ (r'/s')\ =\ (r\isasymcdot s'\ +\ r'\isasymcdot s)\ /\ (s\isasymcdot s')"
\isanewline
\isanewline
\isacommand{lemma}\ mult\_rel\_frac:\isanewline
\ \ \isakeyword{assumes}\ "(r,s)\ \isasymin \ R\ \isasymtimes \ S"\ \isakeyword{and}\ "(r',s')\isasymin \ R\ \isasymtimes \ S"\isanewline
\ \ \isakeyword{shows}\ "mult\_rel\ (r/s)\ (r'/s')\ =\ (r\isasymcdot \ r')\ /\ (s\isasymcdot s')"
\isanewline
\isanewline
\isacommand{definition}\ zero\_rel::"('a\ \isasymtimes \ 'a)\ set"\ \isakeyword{where}\isanewline
\ \ "zero\_rel\ =\ frac\ \isasymzero \ \isasymone "\isanewline
\isanewline
\isacommand{definition}\ one\_rel::"('a\ \isasymtimes \ 'a)\ set"\ \isakeyword{where}\isanewline
\ \ "one\_rel\ =\ frac\ \isasymone \ \isasymone "\isanewline
\isanewline
\isacommand{definition}\isamarkupfalse%
\ carrier{\isacharunderscore}{\kern0pt}quotient{\isacharunderscore}{\kern0pt}ring{\isacharcolon}{\kern0pt}{\isacharcolon}{\kern0pt}\ {\isachardoublequoteopen}{\isacharparenleft}{\kern0pt}{\isacharprime}{\kern0pt}a\ {\isasymtimes}\ {\isacharprime}{\kern0pt}a{\isacharparenright}{\kern0pt}\ set\ set{\isachardoublequoteclose}\isanewline
\ \ \isakeyword{where}\ {\isachardoublequoteopen}carrier{\isacharunderscore}{\kern0pt}quotient{\isacharunderscore}{\kern0pt}ring\ {\isasymequiv}\ rel{\isachardot}{\kern0pt}Partition{\isachardoublequoteclose}\isanewline
\
\isanewline
\isacommand{sublocale}\ comm\_ring\ carrier\_quotient\_ring\ add\_rel\ mult\_rel\ zero\_rel\ one\_rel\isanewline
\isanewline
\isacommand{end}
\end{isabelle}

The code above ends with a proof that the locale \textit{quotient\_ring} is a \textbf{sublocale} of the locale \textit{comm\_ring}, \textit{i.e.}, a proof that the quotient of a commutative ring by a multiplicative submonoid is a commutative ring.
Remember that the complement of a prime ideal $\mathfrak{p}$ in $R$ is a multiplicative submonoid, hence we can apply the previous construction with $S \coloneqq R - \mathfrak{p}$. The resulting ring  $S^{-1} R$ is called the \emph{local ring of $R$ at $\mathfrak{p}$}.

\begin{isabelle}
\isacommand{context}\ pr\_ideal\isanewline
\isakeyword{begin}\isanewline
\isanewline
\isacommand{lemma}\ submonoid\_pr\_ideal:\isanewline
\ \ \isakeyword{shows}\ "submonoid\ (R\ \isasymsetminus \ I)\ R\ (\isasymcdot )\ \isasymone "\isanewline
\isacommand{proof}\isanewline
\ \ \isacommand{show}\ "a\ \isasymcdot \ b\ \isasymin \ R\isasymsetminus I"\ \isakeyword{if}\ "a\ \isasymin \ R\isasymsetminus I"\ "b\ \isasymin \ R\isasymsetminus I"\ \isakeyword{for}\ a\ b\isanewline
\ \ \ \ \isacommand{using}\ that\ \isacommand{by}\ (metis\ Diff\_iff\ absorbent\ comm.multiplicative.composition\_closed)\isanewline
\ \ \isacommand{show}\ "\isasymone \ \isasymin \ R\isasymsetminus I"\isanewline
\ \ \ \ \isacommand{using}\ ideal.ideal(2)\ ideal\_axioms\ pr\_ideal.carrier\_neq\ pr\_ideal\_axioms\ \isanewline
\ \ \ \ \isacommand{by}\ fastforce\isanewline
\isacommand{qed}\ auto%
\isanewline
\isanewline
\isacommand{interpretation}\ local:quotient\_ring\ "(R\ \isasymsetminus \ I)"\ R\ "(+)"\ "(\isasymcdot )"\ \isasymzero \ \isasymone \isanewline
\ \ \isacommand{apply}\ intro\_locales\isanewline
\ \ \isacommand{by}\ (meson\ submonoid\_def\ submonoid\_pr\_ideal)

\isanewline
\isacommand{definition}\ carrier\_local\_ring\_at::\ "('a\ \isasymtimes \ 'a)\ set\ set"\isanewline
\ \ \isakeyword{where}\ "carrier\_local\_ring\_at\ \isasymequiv \ (R\ \isasymsetminus \ I)\isactrlsup \isasyminverse \ R\isactrlbsub (+)\ (\isasymcdot )\ \isasymzero \isactrlesub "\isanewline
\isanewline
\isacommand{definition}\ add\_local\_ring\_at::\ "('a\ \isasymtimes \ 'a)\ set\ \isasymRightarrow \ ('a\ \isasymtimes \ 'a)\ set\ \isasymRightarrow \ ('a\ \isasymtimes \ 'a)\ set"\isanewline
\ \ \isakeyword{where}\ "add\_local\_ring\_at\ \isasymequiv \ local.add\_rel\ "\isanewline
\isanewline
\isacommand{definition}\ mult\_local\_ring\_at::\ "('a\ \isasymtimes \ 'a)\ set\ \isasymRightarrow \ ('a\ \isasymtimes \ 'a)\ set\ \isasymRightarrow \ ('a\ \isasymtimes \ 'a)\ set"\isanewline
\ \ \isakeyword{where}\ "mult\_local\_ring\_at\ \isasymequiv \ local.mult\_rel\ "\isanewline
\isanewline
\isacommand{definition}\ uminus\_local\_ring\_at::\ "('a\ \isasymtimes \ 'a)\ set\ \isasymRightarrow \ ('a\ \isasymtimes \ 'a)\ set"\isanewline
\ \ \isakeyword{where}\ "uminus\_local\_ring\_at\ \isasymequiv \ local.uminus\_rel\ "\isanewline
\isanewline
\isacommand{definition}\ zero\_local\_ring\_at::\ "('a\ \isasymtimes \ 'a)\ set"\isanewline
\ \ \isakeyword{where}\ "zero\_local\_ring\_at\ \isasymequiv \ local.zero\_rel"\isanewline
\isanewline
\isacommand{definition}\ one\_local\_ring\_at::\ "('a\ \isasymtimes \ 'a)\ set"\isanewline
\ \ \isakeyword{where}\ "one\_local\_ring\_at\ \isasymequiv \ local.one\_rel"\isanewline
\isanewline
\isacommand{end}
\end{isabelle}

This allows us to introduce the next construction.

\subsubsection*{The Spectrum of a Ring}

Let $U$ be an open subset of $\text{Spec}\,R$. We define $\mathscr{O}_{\spec} (U)$ to be the set of all functions 
	\[
	s: U \rightarrow \coprod \limits_{\mathfrak{p} \in U} R_{\mathfrak{p}}
	\]
such that $s(\mathfrak{p}) \in R_{\mathfrak{p}}$ for every $\mathfrak{p} \in U$ and such that for every $\mathfrak{p} \in U$, there exist a neighborhood $V$ of $\mathfrak{p}$, contained in $U$, and elements $r, f \in R$, such that for each $\mathfrak{q} \in V$, $f \notin \mathfrak{q}$ and $s(\mathfrak{q}) = r/f$.
We can prove that $\mathscr{O}_{\spec} (U)$ is closed under the addition and multiplication of such functions and we take for the additive unit (\textit{resp.} multiplicative unit) the function that maps each $\mathfrak{p}$ on the additive unit (\textit{resp.} multiplicative unit) of $R_{\mathfrak{p}}$.
Last, if $V \subseteq U$, then the morphism of rings $\mathscr{O}_{\spec} (U) \rightarrow \mathscr{O}_{\spec} (V)$ maps each $s \in \mathscr{O}_{\spec} (U)$ on its restriction to $V$.
One checks $\mathscr{O}_{\spec} $ is a sheaf of rings on $\text{Spec}\,R$ and the pair $(\text{Spec}\,R, \mathscr{O}_{\spec} )$ is called the \emph{spectrum of the ring $R$}.
This last construction translates smoothly in Isabelle. 

\begin{isabelle}
\isacommand{definition}\ (\isakeyword{in}\ comm\_ring)\ \isanewline
\ \ \ \ is\_locally\_frac::\ "('a\ set\ \isasymRightarrow \ ('a\ \isasymtimes \ 'a)\ set)\ \isasymRightarrow \ 'a\ set\ set\ \isasymRightarrow \ bool"\isanewline
\ \ \ \ \isakeyword{where}\ "is\_locally\_frac\ s\ V\ \isasymequiv \ (\isasymexists r\ f.\ r\ \isasymin \ R\ \isasymand \ f\ \isasymin \ R\ \isasymand \ (\isasymforall \isasymqq \ \isasymin \ V.\ \isanewline
\ \ \ \ \ \ \ \ \ \ \ \ f\ \isasymnotin \ \isasymqq \ \isasymand \ s\ \isasymqq \ =\ quotient\_ring.frac\ (R\ \isasymsetminus \ \isasymqq )\ R\ (+)\ (\isasymcdot )\ \isasymzero \ r\ f))"\isanewline
\isanewline
\isacommand{definition}\ (\isakeyword{in}\ comm\_ring)\ is\_regular\isanewline
\ \ \ \ \ \ \ \ \ \ \ \ \ \ ::\ "('a\ set\ \isasymRightarrow \ ('a\ \isasymtimes \ 'a)\ set)\ \isasymRightarrow \ 'a\ set\ set\ \isasymRightarrow \ bool"\isanewline
\ \ \isakeyword{where}\ "is\_regular\ s\ U\ \isasymequiv \ \isasymforall \isasympp .\ \isasympp \ \isasymin \ U\ \isasymlongrightarrow \ \isanewline
\ \ (\isasymexists V.\ is\_zariski\_open\ V\ \isasymand \ V\ \isasymsubseteq \ U\ \isasymand \ \isasympp \ \isasymin \ V\ \isasymand \ (is\_locally\_frac\ s\ V))"\isanewline
\isanewline
\isacommand{definition}\ (\isakeyword{in}\ comm\_ring)\ sheaf\_spec\isanewline
\ \ \ \ ::\ "'a\ set\ set\ \isasymRightarrow \ ('a\ set\ \isasymRightarrow \ ('a\ \isasymtimes \ 'a)\ set)\ set"\ ("\isasymO \ \_"\ [90]\ 90)\isanewline
\ \ \isakeyword{where}\ "\isasymO \ U\ \isasymequiv \ \{s\isasymin (\isasymPi \isactrlsub E\ \isasympp \isasymin U.\ (R\isactrlbsub \isasympp \ (+)\ (\isasymcdot )\ \isasymzero \isactrlesub )).\ is\_regular\ s\ U\}"\isanewline
\isanewline
\isacommand{definition}\ (\isakeyword{in}\ comm\_ring)\ add\_sheaf\_spec\isanewline
\ \ \ \ ::\ "('a\ set)\ set\ \isasymRightarrow \ ('a\ set\ \isasymRightarrow \ ('a\ \isasymtimes \ 'a)\ set)\ \isanewline
\ \ \ \ \ \ \ \ \ \ \isasymRightarrow \ ('a\ set\ \isasymRightarrow \ ('a\ \isasymtimes \ 'a)\ set)\ \isasymRightarrow \ ('a\ set\ \isasymRightarrow \ ('a\ \isasymtimes \ 'a)\ set)"\isanewline
\ \ \ \ \isakeyword{where}\ "add\_sheaf\_spec\ U\ s\ s'\ \isasymequiv \ \isanewline
\ \ \ \ \ \ \ \ \ \ \isasymlambda \isasympp \isasymin U.\ quotient\_ring.add\_rel\ (R\ \isasymsetminus \ \isasympp )\ R\ (+)\ (\isasymcdot )\ \isasymzero \ (s\ \isasympp )\ (s'\ \isasympp )"\isanewline
\isanewline
\isacommand{definition}\ (\isakeyword{in}\ comm\_ring)\ mult\_sheaf\_spec\isanewline
\ \ \ \ ::\ "('a\ set)\ set\ \isasymRightarrow \ ('a\ set\ \isasymRightarrow \ ('a\ \isasymtimes \ 'a)\ set)\ \isanewline
\ \ \ \ \ \ \ \ \ \ \ \ \isasymRightarrow \ ('a\ set\ \isasymRightarrow \ ('a\ \isasymtimes \ 'a)\ set)\ \isasymRightarrow \ ('a\ set\ \isasymRightarrow \ ('a\ \isasymtimes \ 'a)\ set)"\isanewline
\ \ \ \ \isakeyword{where}\ "mult\_sheaf\_spec\ U\ s\ s'\ \isasymequiv \ \isanewline
\ \ \ \ \ \ \ \ \ \isasymlambda \isasympp \isasymin U.\ quotient\_ring.mult\_rel\ (R\ \isasymsetminus \ \isasympp )\ R\ (+)\ (\isasymcdot )\ \isasymzero \ (s\ \isasympp )\ (s'\ \isasympp )"\isanewline
\isanewline
\isacommand{definition}\ (\isakeyword{in}\ comm\_ring)\ zero\_sheaf\_spec\isanewline
\ \ \ \ ::\ "'a\ set\ set\ \isasymRightarrow \ ('a\ set\ \isasymRightarrow \ ('a\ \isasymtimes \ 'a)\ set)"\isanewline
\ \ \ \ \isakeyword{where}\ "zero\_sheaf\_spec\ U\ \isasymequiv \ \isanewline
\ \ \ \ \ \ \ \ \ \ \ \ \ \ \ \ \ \ \ \isasymlambda \isasympp \isasymin U.\ quotient\_ring.zero\_rel\ (R\ \isasymsetminus \ \isasympp )\ R\ (+)\ (\isasymcdot )\ \isasymzero \ \isasymone "\isanewline
\isanewline
\isacommand{definition}\ (\isakeyword{in}\ comm\_ring)\ one\_sheaf\_spec\isanewline
\ \ \ \ ::\ "'a\ set\ set\ \isasymRightarrow \ ('a\ set\ \isasymRightarrow \ ('a\ \isasymtimes \ 'a)\ set)"\isanewline
\ \ \ \ \isakeyword{where}\ "one\_sheaf\_spec\ U\ \isasymequiv \ \isanewline
\ \ \ \ \ \ \ \ \ \ \ \ \ \ \ \ \ \ \ \ \isasymlambda \isasympp \isasymin U.\ quotient\_ring.one\_rel\ (R\ \isasymsetminus \ \isasympp )\ R\ (+)\ (\isasymcdot )\ \isasymzero \ \isasymone "\isanewline
\isanewline
\isacommand{definition}\ (\isakeyword{in}\ comm\_ring)\ sheaf\_spec\_morphisms\isanewline
\ \ \ \ \ \ \ \ ::\ "'a\ set\ set\ \isasymRightarrow \ 'a\ set\ set\ \isasymRightarrow \ (('a\ set\ \isasymRightarrow \ ('a\ \isasymtimes \ 'a)\ set)\ \isanewline
\ \ \ \ \ \ \ \ \ \ \ \ \ \ \ \ \ \ \ \ \ \ \ \ \ \ \ \ \ \ \ \ \ \ \ \ \ \ \ \ \isasymRightarrow \ ('a\ set\ \isasymRightarrow \ ('a\ \isasymtimes \ 'a)\ set))"\isanewline
\ \ \ \ \ \ \ \ \isakeyword{where}\ "sheaf\_spec\_morphisms\ U\ V\ \isasymequiv \ \isasymlambda s\isasymin (\isasymO \ U).\ restrict\ s\ V"\isanewline
\isanewline
\isacommand{lemma}\ (\isakeyword{in}\ comm\_ring)\ sheaf\_spec\_is\_sheaf:\isanewline
\ \ \isakeyword{shows}\ "sheaf\_of\_rings\ Spec\ is\_zariski\_open\ sheaf\_spec\ \isanewline
\ \ \ \ \ \ \ \ \ \ \ \ sheaf\_spec\_morphisms\ \isasymO b\ add\_sheaf\_spec\ mult\_sheaf\_spec\ \isanewline
\ \ \ \ \ \ \ \ \ \ \ \ \ \ \ \ \ \ \ \ \ \ \ \ \ \ \ \ \ \ \ \ \ \ \ \ \ \ \ \ zero\_sheaf\_spec\ one\_sheaf\_spec"
\end{isabelle}

Next, we introduce ringed spaces and their morphisms.

\subsubsection*{Ringed Spaces}

\begin{definition}[ringed space]
	A ringed space is a pair $(X, \mathscr{O}_X)$, where $X$ is a topological space and $\mathscr{O}_X$ is a sheaf of rings on $X$.
\end{definition}

\begin{isabelle}
\isacommand{locale}\ ringed\_space\ =\ sheaf\_of\_rings
\end{isabelle}

\begin{definition}[morphism of ringed spaces]
	A morphism of ringed spaces from $(X, \mathscr{O}_X)$ to $(Y, \mathscr{O}_Y)$ is a pair $(f, \phi_f)$ consisting of a continuous map $f: X \rightarrow Y$ between topological spaces and a morphism $\phi_f: \mathscr{O}_Y \rightarrow f_{*} \mathscr{O}_X$ of sheaves of rings.   
\end{definition}

\begin{isabelle}
\isacommand{locale}\ morphism\_ringed\_spaces\ =\isanewline
\ \ \ \ \ \ im\_sheaf\ X\ is\_open\isactrlsub X\ \isasymO \isactrlsub X\ \isasymrho \isactrlsub X\ b\ add\_str\isactrlsub X\ mult\_str\isactrlsub X\ zero\_str\isactrlsub X\ one\_str\isactrlsub X\ \isanewline
\ \ \ \ \ \ \ \ \ \ \ \ \ \ \ \ \ \ \ \ \ \ \ \ \ \ \ \ \ \ \ \ \ \ \ \ \ \ \ \ \ \ \ \ \ \ \ \ \ \ \ \ \ \ \ \ Y\ is\_open\isactrlsub Y\ f\ \isanewline
\ \ \ \ \ \ +\ codom:\ ringed\_space\ Y\ is\_open\isactrlsub Y\ \isasymO \isactrlsub Y\ \isasymrho \isactrlsub Y\ d\ add\_str\isactrlsub Y\ mult\_str\isactrlsub Y\ \isanewline
\ \ \ \ \ \ \ \ \ \ \ \ \ \ \ \ \ \ \ \ \ \ \ \ \ \ \ \ \ \ \ \ \ \ \ \ \ \ \ \ \ \ \ \ \ \ \ \ \ \ \ \ zero\_str\isactrlsub Y\ one\_str\isactrlsub Y\isanewline
\ \ \ \ \ \ \isakeyword{for}\ X\ \isakeyword{and}\ is\_open\isactrlsub X\ \isakeyword{and}\ \isasymO \isactrlsub X\ \isakeyword{and}\ \isasymrho \isactrlsub X\ \isakeyword{and}\ b\ \isakeyword{and}\ add\_str\isactrlsub X\ \isakeyword{and}\ mult\_str\isactrlsub X\isanewline
\ \ \ \ \ \ \ \ \ \ \isakeyword{and}\ zero\_str\isactrlsub X\ \isakeyword{and}\ one\_str\isactrlsub X\ \isakeyword{and}\ Y\ \isakeyword{and}\ is\_open\isactrlsub Y\ \isakeyword{and}\ \isasymO \isactrlsub Y\ \isakeyword{and}\ \isasymrho \isactrlsub Y\ \isanewline
\ \ \ \ \ \ \ \ \ \ \isakeyword{and}\ d\ \isakeyword{and}\ add\_str\isactrlsub Y\ \isakeyword{and}\ mult\_str\isactrlsub Y\ \isakeyword{and}\ zero\_str\isactrlsub Y\ \isakeyword{and}\ one\_str\isactrlsub Y\isanewline
\ \ \ \ \ \ \ \ \ \ \isakeyword{and}\ f\ \isanewline
\ \ \ \ \ \ +\ \isakeyword{fixes}\ \isasymphi \isactrlsub f::\ "'c\ set\ \isasymRightarrow \ ('d\ \isasymRightarrow \ 'b)"\isanewline
\ \ \ \ \isakeyword{assumes}\ is\_morphism\_of\_sheaves:\ "morphism\_sheaves\_of\_rings\ Y\ \isanewline
\ \ \ \ \ \ \ \ \ \ \ \ \ \ \ \ is\_open\isactrlsub Y\ \isasymO \isactrlsub Y\ \isasymrho \isactrlsub Y\ d\ add\_str\isactrlsub Y\ mult\_str\isactrlsub Y\ zero\_str\isactrlsub Y\ one\_str\isactrlsub Y\isanewline
\ \ \ \ \ \ \ \ \ \ \ \ \ \ \ \ im\_sheaf\ im\_sheaf\_morphisms\ b\ add\_im\_sheaf\ mult\_im\_sheaf\ \isanewline
\ \ \ \ \ \ \ \ \ \ \ \ \ \ \ \ zero\_im\_sheaf\ one\_im\_sheaf\ \isasymphi \isactrlsub f"
\end{isabelle}

A classic and ubiquitous notion in mathematics is the \emph{direct limit} of a family of structured sets, here a family of rings.\footnote{In more modern parlance it is called a \textit{colimit}.}

\subsubsection*{Direct Limit of a Presheaf of Rings}

Let $X$ be a topological space and $\mathscr{F}$ a presheaf of rings on $X$. The set $I$ will denote a subset of the set of all open subsets of $X$ such that for every $U$ and $V$ in $I$ there exists $W$ in $I$ with $W \subseteq U \cap V$.
Given $U, V \in I$, $s \in \mathscr{F}(U)$ and $t \in \mathscr{F}(V)$,  we say that 
$s \sim t$ if and only if there exists $W \in I$ such that $W \subseteq U \cap V$ and $s \restriction W = t \restriction W$. One checks that $\sim$ is an equivalence relation.
Now, we consider the quotient of the disjoint union of the $\mathscr{F}(U)$'s.
	\[
	\varinjlim_I \mathscr{F} \coloneqq \coprod_{U \in I} \mathscr{F}(U) \bigg/ \sim 
	\]
Last, we define the following binary operations on $\displaystyle \varinjlim_{I} \mathscr{F}$:
	\begin{align*}
	[(U, s)] + [(V, t)] & = [(W, s \restriction W + t \restriction W)] \\
	[(U, s)] \times [(V, t)] & = [(W, s \restriction W \times t \restriction W)]
	\end{align*}
for some $W \in I$ such that $W \subseteq U \cap V$ and where the symbol $[\_]$ denotes the equivalence class of an element.
These operations are well-defined, and assuming $V \in I$ one can prove that  
	\[
	(\varinjlim_I \mathscr{F}, [(V, 0_V)], +, [(V, 1_V)], \times)
	\]
is a ring. 

\begin{definition}[direct limit of a presheaf of rings]
	Let $X$ be a topological space, $\mathscr{F}$ a presheaf of rings on $X$ and $I$ a set of open subsets of $X$. The direct limit of $\mathscr{F}$ over $I$ is the ring $\displaystyle \varinjlim_I \mathscr{F}$, denoted simply $\varinjlim \mathscr{F}$ if $I$ is clear from the context.
\end{definition}

The definition above, which could naively appear as a dependent type, gives rise to a smooth translation in Isabelle using the locale mechanism.

\begin{isabelle}
\isacommand{locale}\ direct\_lim\ =\ sheaf\_of\_rings\ +\isanewline
\ \ \isakeyword{fixes}\ I::\ "'a\ set\ set"\isanewline
\ \ \isakeyword{assumes}\ subset\_of\_opens:\ "\isasymAnd U.\ U\ \isasymin \ I\ \isasymLongrightarrow \ is\_open\ U"\isanewline
\ \ \ \ \isakeyword{and}\ has\_lower\_bound:\ "\isasymAnd U\ V.\ \isasymlbrakk \ U\isasymin I;\ V\isasymin I\ \isasymrbrakk \ \isasymLongrightarrow \ \isasymexists W\isasymin I.\ W\ \isasymsubseteq \ U\ \isasyminter \ V"\isanewline
\isakeyword{begin}\isanewline
\isanewline
\isacommand{definition}\ rel::\ "('a\ set\ \isasymtimes \ 'b)\ \isasymRightarrow \ ('a\ set\ \isasymtimes \ 'b)\ \isasymRightarrow \ bool"\ (\isakeyword{infix}\ "\isasymsim "\ 80)\isanewline
\ \ \isakeyword{where}\ "x\ \isasymsim \ y\ \isasymequiv \ (fst\ x\ \isasymin \ I\ \isasymand \ fst\ y\ \isasymin \ I)\ \isasymand \ (snd\ x\ \isasymin \ \isasymFF \ (fst\ x)\ \isanewline
\ \ \ \ \ \ \ \ \ \ \ \ \ \ \ \ \ \isasymand \ snd\ y\ \isasymin \ \isasymFF \ (fst\ y))\ \isasymand \ (\isasymexists W.\ (W\ \isasymin \ I)\ \isanewline
\ \ \ \ \ \ \ \ \ \ \ \ \ \ \ \ \ \isasymand \ (W\ \isasymsubseteq \ fst\ x\ \isasyminter \ fst\ y)\ \isasymand \ \isasymrho \ (fst\ x)\ W\ (snd\ x)\ \isanewline
\ \ \ \ \ \ \ \ \ \ \ \ \ \ \ \ \ \ \ \ \ \ \ \ \ \ \ \ \ \ \ \ \ \ \ \ \ \ \ \ \ \ \ \ \ \ \ \ \ =\ \isasymrho \ (fst\ y)\ W\ (snd\ y))"\isanewline
\isanewline
\isacommand{interpretation}\ rel:equivalence\ "(Sigma\ I\ \isasymFF )"\ "\{(x,\ y).\ x\ \isasymsim \ y\}"\isanewline
\isanewline
\isacommand{definition}\ class\_of::\ "'a\ set\ \isasymRightarrow \ 'b\ \isasymRightarrow \ ('a\ set\ \isasymtimes \ 'b)\ set"\ ("\isasymlfloor (\_,/\ \_)\isasymrfloor ")\isanewline
\ \ \isakeyword{where}\ "\isasymlfloor U,s\isasymrfloor \ \isasymequiv \ rel.Class\ (U,\ s)"\isanewline
\isanewline
\isacommand{definition}\ carrier\_direct\_lim::\ "('a\ set\ \isasymtimes \ 'b)\ set\ set"\isanewline
\ \ \isakeyword{where}\ "carrier\_direct\_lim\ \isasymequiv \ rel.Partition"\isanewline
\isanewline
\isacommand{definition}\ get\_lower\_bound::\ "'a\ set\ \isasymRightarrow \ 'a\ set\ \isasymRightarrow \ 'a\ set"\ \isakeyword{where}\isanewline
\ \ "get\_lower\_bound\ U\ V =\ (SOME\ W.\ W\ \isasymin \ I\ \isasymand \ W\ \isasymsubseteq \ U\ \isasymand \ W\ \isasymsubseteq \ V)"\isanewline
\isanewline
\isacommand{definition}\ add\_rel\ ::\ "('a\ set\ \isasymtimes \ 'b)\ set\ \isasymRightarrow \ ('a\ set\ \isasymtimes \ 'b)\ set\ \isanewline
\ \ \ \ \ \ \ \ \ \ \ \ \ \ \ \ \ \ \ \ \ \ \ \ \ \ \ \ \ \ \ \ \ \ \ \ \ \ \ \ \ \ \ \ \ \ \ \ \isasymRightarrow \ ('a\ set\ \isasymtimes \ 'b)\ set"\isanewline
\ \ \isakeyword{where}\ "add\_rel\ X\ Y\ \isasymequiv \ let\isanewline
\ \ \ \ \ \ \ \ \ \ \ \ \ \ x\ =\ (SOME\ x.\ x\ \isasymin \ X);\isanewline
\ \ \ \ \ \ \ \ \ \ \ \ \ \ y\ =\ (SOME\ y.\ y\ \isasymin \ Y);\isanewline
\ \ \ \ \ \ \ \ \ \ \ \ \ \ w\ =\ get\_lower\_bound\ (fst\ x)\ (fst\ y)\isanewline
\ \ \ \ \ \ \ \ \ \ \ \ in\isanewline
\ \ \ \ \ \ \ \ \ \ \ \ \ \ \isasymlfloor w,\ add\_str\ w\ (\isasymrho \ (fst\ x)\ w\ (snd\ x))\ (\isasymrho \ (fst\ y)\ w\ (snd\ y))\isasymrfloor "\isanewline
\isanewline
\isacommand{definition}\ mult\_rel\ ::\ "('a\ set\ \isasymtimes \ 'b)\ set\ \isasymRightarrow \ ('a\ set\ \isasymtimes \ 'b)\ set\ \isanewline
\ \ \ \ \ \ \ \ \ \ \ \ \ \ \ \ \ \ \ \ \ \ \ \ \ \ \ \ \ \ \ \ \ \ \ \ \ \ \ \ \ \ \ \ \ \ \ \ \ \ \ \ \isasymRightarrow \ ('a\ set\ \isasymtimes \ 'b)\ set"\isanewline
\ \ \isakeyword{where}\ "mult\_rel\ X\ Y\ \isasymequiv \ let\isanewline
\ \ \ \ \ \ \ \ \ \ \ \ \ \ x\ =\ (SOME\ x.\ x\ \isasymin \ X);\isanewline
\ \ \ \ \ \ \ \ \ \ \ \ \ \ y\ =\ (SOME\ y.\ y\ \isasymin \ Y);\isanewline
\ \ \ \ \ \ \ \ \ \ \ \ \ \ w\ =\ get\_lower\_bound\ (fst\ x)\ (fst\ y)\isanewline
\ \ \ \ \ \ \ \ \ \ \ \ in\isanewline
\ \ \ \ \ \ \ \ \ \ \ \ \ \ \isasymlfloor w,\ mult\_str\ w\ (\isasymrho \ (fst\ x)\ w\ (snd\ x))\ (\isasymrho \ (fst\ y)\ w\ (snd\ y))\isasymrfloor "\isanewline
\isanewline
\isacommand{lemma}\ direct\_{lim}\_{is}\_{ring}:\isanewline
\ \ \isakeyword{assumes}\ "U\ \isasymin \ I"\isanewline
\ \ \isakeyword{shows}\ "ring\ carrier\_direct\_lim\ add\_rel\ mult\_rel\ \isasymlfloor U,\ \isasymzero \isactrlbsub U\isactrlesub \isasymrfloor \ \isasymlfloor U,\ \isasymone \isactrlbsub U\isactrlesub \isasymrfloor "\isanewline
\isanewline
\isacommand{end}
\end{isabelle}

For every $U \in I$, there is a canonical map from $\mathscr{F}(U)$ to $\varinjlim \mathscr{F}$. 

\begin{isabelle}
\isacommand{definition}\ (\isakeyword{in}\ direct\_lim)\ \isanewline
\ \ \ \ \ \ \ \ canonical\_fun\ ::\ "'a\ set\ \isasymRightarrow \ 'b\ \isasymRightarrow \ ('a\ set\ \isasymtimes \ 'b)\ set"\isanewline
\ \ \isakeyword{where}\ "canonical\_fun\ U\ x\ =\ \isasymlfloor U,\ x\isasymrfloor "
\end{isabelle}

The direct limit of a presheaf of rings satisfies a useful universal property. Indeed, for every ring $A$ equipped with ring morphisms $\psi_U: \mathscr{F}(U) \rightarrow A$ for $U \in I$ satisfying $\psi_V \circ \rho_{UV} = \psi_U$ for every inclusion $V \subseteq U$, there exists a unique ring morphism $u$ making the following diagram commute.
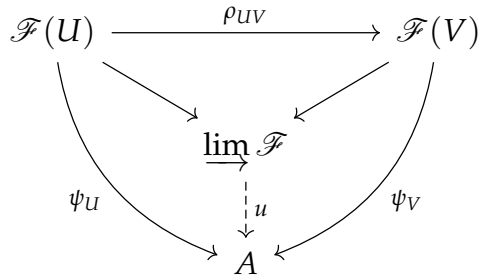
\[
\begin{tikzcd}
\mathscr{F}(U)  \arrow[rr, "\rho_{UV}"] \arrow[rd] \arrow[rdd, bend right, "\psi_U"'] & & \mathscr{F}(V) \arrow[ld]	\arrow[ldd, bend left, "\psi_V"]	\\
						& \varinjlim \mathscr{F} \arrow[d, "u", dashed] &	\\
						& A &  
\end{tikzcd}
\] 
The counterpart in Isabelle is as follows.

\begin{isabelle}
\isacommand{proposition}\ (\isakeyword{in}\ direct\_lim)\ universal\_property:\isanewline
\ \ \isakeyword{fixes}\ A::\ "'c\ set"\ \isakeyword{and}\ \isasympsi ::\ "'a\ set\ \isasymRightarrow \ ('b\ \isasymRightarrow \ 'c)"\ \isanewline
\ \ \ \ \isakeyword{and}\ add::\ "'c\ \isasymRightarrow \ 'c\ \isasymRightarrow \ 'c"\ \isakeyword{and}\ mult::\ "'c\ \isasymRightarrow \ 'c\ \isasymRightarrow \ 'c"\ \isanewline
\ \ \ \ \isakeyword{and}\ zero::\ "'c"\ \isakeyword{and}\ one::\ "'c"\isanewline
\ \ \isakeyword{assumes}\ "ring\ A\ add\ mult\ zero\ one"\isanewline
\ \ \ \ \isakeyword{and}\ "\isasymAnd U.\ U\ \isasymin \ I\ \isasymLongrightarrow \ ring\_homomorphism\ (\isasympsi \ U)\ (\isasymFF \ U)\ (+\isactrlbsub U\isactrlesub )\ \isanewline
\ \ \ \ \ \ \ \ \ \ \ \ \ \ \ \ \ \ \ \ \ \ \ \ \ \ \ \ \ \ \ \ \ \ (\isasymcdot \isactrlbsub U\isactrlesub )\ \isasymzero \isactrlbsub U\isactrlesub \ \isasymone \isactrlbsub U\isactrlesub \ A\ add\ mult\ zero\ one"\isanewline
\ \ \ \ \isakeyword{and}\ "\isasymAnd U\ V\ x.\ \isasymlbrakk U\ \isasymin \ I;\ V\ \isasymin \ I;\ V\ \isasymsubseteq \ U;\ x\ \isasymin \ (\isasymFF \ U)\isasymrbrakk \ \isanewline
\ \ \ \ \ \ \ \ \ \ \ \ \ \ \ \ \ \ \ \ \ \ \ \ \ \ \ \ \ \ \ \ \ \ \ \ \ \ \ \ \ \ \isasymLongrightarrow \ (\isasympsi \ V\ \isasymcirc \ \isasymrho \ U\ V)\ x\ =\ \isasympsi \ U\ x"\isanewline
\ \ \isakeyword{shows}\ "\isasymforall V\isasymin I.\ \isasymexists !u.\ ring\_homomorphism\ u\ carrier\_direct\_lim\ add\_rel\ \isanewline
\ \ \ \ \ \ \ \ \ \ \ \ \ \ \ \ \ \ \ \ \ \ \ \ \ mult\_rel\ \isasymlfloor V,\isasymzero \isactrlbsub V\isactrlesub \isasymrfloor \ \isasymlfloor V,\isasymone \isactrlbsub V\isactrlesub \isasymrfloor \ A\ add\ mult\ zero\ one\isanewline
\ \ \ \ \ \ \ \ \ \ \ \ \ \ \ \ \ \ \ \isasymand \ (\isasymforall U\isasymin I.\ \isasymforall x\isasymin (\isasymFF \ U).\ (u\ \isasymcirc \ canonical\_fun\ U)\ x\ =\ \isasympsi \ U\ x)"
\end{isabelle}

One can instantiate direct limits in the following particular case.

\subsubsection*{Stalks of a Presheaf of Rings}

\begin{definition}[stalks of a presheaf]
	Let $X$ be a topological space, $\mathscr{F}$ a presheaf of rings on $X$ and $x$ an element of $X$.
	The stalk $\mathscr{F}_x$ of $\mathscr{F}$ at $x$ is the direct limit of $\mathscr{F}$ over the set of all the open neighborhoods of $x$.   
\end{definition}

\begin{isabelle}
\isacommand{locale}\ stalk\ =\ direct\_lim\ +\isanewline
\ \ \isakeyword{fixes}\ x::\ "'a"\isanewline
\ \ \isakeyword{assumes}\ is\_elem:\ "x\ \isasymin \ S"\ \isakeyword{and}\ index:\ "I\ =\ \{U.\ is\_open\ U\ \isasymand \ x\ \isasymin \ U\}"\isanewline
\isakeyword{begin}\isanewline
\isanewline
\isacommand{definition}\ carrier\_stalk::\ "('a\ set\ \isasymtimes \ 'b)\ set\ set"\isanewline
\ \ \isakeyword{where}\ "carrier\_stalk\ \isasymequiv \ dlim\ \isasymFF \ \isasymrho \ (neighborhoods\ x)"\isanewline
\isanewline
\isacommand{definition}\ add\_stalk\isanewline
\ \ \ \ ::\ "('a\ set\ \isasymtimes \ 'b)\ set\ \isasymRightarrow \ ('a\ set\ \isasymtimes \ 'b)\ set\ \isasymRightarrow \ ('a\ set\ \isasymtimes \ 'b)\ set"\isanewline
\ \ \isakeyword{where}\ "add\_stalk\ \isasymequiv \ add\_rel"\isanewline
\isanewline
\isacommand{definition}\ mult\_stalk\isanewline
\ \ \ \ ::\ "('a\ set\ \isasymtimes \ 'b)\ set\ \isasymRightarrow \ ('a\ set\ \isasymtimes \ 'b)\ set\ \isasymRightarrow \ ('a\ set\ \isasymtimes \ 'b)\ set"\isanewline
\ \ \isakeyword{where}\ "mult\_stalk\ \isasymequiv \ mult\_rel"\isanewline
\isanewline
\isacommand{definition}\ zero\_stalk::\ "'a\ set\ \isasymRightarrow \ ('a\ set\ \isasymtimes \ 'b)\ set"\isanewline
\ \ \isakeyword{where}\ "zero\_stalk\ V\ \isasymequiv \ class\_of\ V\ \isasymzero \isactrlbsub V\isactrlesub "\isanewline
\isanewline
\isacommand{definition}\ one\_stalk::\ "'a\ set\ \isasymRightarrow \ ('a\ set\ \isasymtimes \ 'b)\ set"\isanewline
\ \ \isakeyword{where}\ "one\_stalk\ V\ \isasymequiv \ class\_of\ V\ \isasymone \isactrlbsub V\isactrlesub "\isanewline
\isanewline
\isacommand{end}
\end{isabelle}

One further step towards schemes consists in introducing local rings.

\subsubsection*{Local Rings}

 First, one needs to teach Isabelle simple facts about \emph{maximal ideals}, starting with their definition. 

\begin{definition}[maximal ideal]
	Let $\mathfrak{m}$ be an ideal of $R$. The ideal $\mathfrak{m}$ is maximal if $\mathfrak{m} \neq R$ and there is no ideal $\mathfrak{a} \neq R$ such that $\mathfrak{m} \subset \mathfrak{a}$.
\end{definition}

\begin{isabelle}
\isacommand{locale}\ max\_ideal\ =\ comm\_ring\ R\ "(+)"\ "(\isasymcdot )"\ "\isasymzero "\ "\isasymone "\ \isanewline
\ \ \ \ \ \ \ \ \ \ \ \ \ \ \ \ \ \ \ \ \ \ \ \ \ \ \ \ \ \ \ \ \ \ +\ ideal\ I\ \ R\ "(+)"\ "(\isasymcdot )"\ "\isasymzero "\ "\isasymone "\isanewline
\ \ \isakeyword{for}\ R\ \isakeyword{and}\ I\ \isakeyword{and}\ addition\ (\isakeyword{infixl}\ "+"\ 65)\ \isanewline
\ \ \ \ \isakeyword{and}\ multiplication\ (\isakeyword{infixl}\ "\isasymcdot "\ 70)\ \isakeyword{and}\ zero\ ("\isasymzero ")\ \isakeyword{and}\ unit\ ("\isasymone ")\ +\isanewline
\ \ \isakeyword{assumes}\ neq\_ring:\ "I\ \isasymnoteq \ R"\ \isanewline
\ \ \ \ \ \ \ \isakeyword{and}\ is\_max:\ "\isasymAnd \isasymaa .\ \isasymlbrakk ideal\ \isasymaa \ R\ (+)\ (\isasymcdot )\ \isasymzero \ \isasymone ;\isasymaa \ \isasymnoteq \ R;I\ \isasymsubseteq \ \isasymaa \ \isasymrbrakk \isasymLongrightarrow \ I\ =\ \isasymaa "
\end{isabelle}

Classic facts about maximal ideals include the fact that every maximal ideal is prime.

\begin{isabelle}
\isacommand{proposition}\ (\isakeyword{in}\ max\_ideal)\ is\_pr\_ideal:\ "pr\_ideal\ R\ I\ (+)\ (\isasymcdot )\ \isasymzero \ \isasymone "
\end{isabelle}
	
We then introduce the notion of a \emph{local ring}.

\begin{definition}[local ring]
	A ring is a local ring if it has a unique maximal left ideal. 	
\end{definition}

\begin{isabelle}
\isacommand{locale}\ local\_ring\ =\ ring\ +\isanewline
\ \ \isakeyword{assumes}\ is\_unique:\ "\isasymAnd I\ J.\ max\_lideal\ I\ R\ (+)\ (\isasymcdot )\ \isasymzero \ \isasymone \isanewline
\ \ \ \ \ \ \ \ \ \ \ \ \ \ \ \ \ \ \ \ \ \ \ \ \ \ \ \ \ \isasymLongrightarrow \ max\_lideal\ J\ R\ (+)\ (\isasymcdot )\ \isasymzero \ \isasymone \ \isasymLongrightarrow \ I\ =\ J"\isanewline
\ \ \ \ \isakeyword{and}\ has\_max\_lideal:\ "\isasymexists \isasymww .\ max\_lideal\ \isasymww \ R\ (+)\ (\isasymcdot )\ \isasymzero \ \isasymone "
\end{isabelle}

Two remarks are in order. First, note that the property of being a local ring does not require the ring to be commutative. As a consequence, for generality, we do not assume the commutativity of the ring in the above definition and in its locale counterpart: hence the use of left ideals instead of two-sided ideals. Using right ideals would give rise to an equivalent definition. 
Second, we recall the following result which can be obtained using Zorn's lemma. Let $R$ be a unital ring, every proper ideal of $R$ lies in a maximal ideal of $R$. As a corollary, every nonzero unital ring has a maximal ideal, hence we could dispense with the property \textit{has\_max\_lideal} in the locale \textit{local\_ring}, but only by adding the extra assumption that $R$ is nonzero. Moreover, our formulation of this locale works as a convenient shortcut, avoiding the cost of formalizing the result recalled above.      

We prove as expected that the local ring of $R$ at $\mathfrak{p}$, previously denoted $R_{\mathfrak{p}}$, is a local ring.

\begin{isabelle}
\isacommand{lemma}\ (\isakeyword{in}\ pr\_ideal)\ local\_ring\_at\_is\_local:\isanewline
\ \ \isakeyword{shows}\ "local\_ring\ carrier\_local\_ring\_at\ add\_local\_ring\_at\ \isanewline
\ \ \ \ \ \ \ \ \ \ \ \ \ \ mult\_local\_ring\_at\ zero\_local\_ring\_at\ one\_local\_ring\_at"
\end{isabelle}

Local rings have their corresponding notion of morphisms. 

\begin{definition}[local homomorphism]
	Let $A$ and $B$ be two local rings and $\mathfrak{m}_A$, $\mathfrak{m}_B$ their respective maximal ideals.
	A local homomorphism $f: A \rightarrow B$ is a morphism of rings such that $\mathfrak{m}_A \subseteq f^{-1} (\mathfrak{m}_B)$. 
\end{definition}

\begin{rem}
	The inclusion $\mathfrak{m}_A \subseteq f^{-1} (\mathfrak{m}_B)$ implies the equality $f^{-1} (\mathfrak{m}_B) = \mathfrak{m}_A$.
\end{rem}	

\begin{isabelle}
\isacommand{locale}\ local\_ring\_morphism\ =\isanewline
\ \ \ \ \ \ \ \ \ \ source:\ local\_ring\ A\ "(+)"\ "(\isasymcdot )"\ \isasymzero \ \isasymone \ \isanewline
\ \ \ \ \ \ \ \ \ \ +\ target:\ local\_ring\ B\ "(+')"\ "(\isasymcdot ')"\ "\isasymzero '"\ "\isasymone '"\isanewline
\ \ \ \ \ \ \ \ \ \ +\ ring\_homomorphism\ f\ A\ "(+)"\ "(\isasymcdot )"\ "\isasymzero "\ "\isasymone "\ B\ "(+')"\ "(\isasymcdot ')"\ "\isasymzero '"\ "\isasymone '"\isanewline
\ \ \isakeyword{for}\ f\ \isakeyword{and}\ A\ \isakeyword{and}\ addition\ (\isakeyword{infixl}\ "+"\ 65)\ \isakeyword{and}\ multiplication\ (\isakeyword{infixl}\ "\isasymcdot "\ 70)\ \isanewline
\ \ \ \ \ \isakeyword{and}\ zero\ ("\isasymzero ")\ \isakeyword{and}\ unit\ ("\isasymone ")\ \isakeyword{and}\ B\ \isakeyword{and}\ addition'\ (\isakeyword{infixl}\ "+'{\kern0pt}'"\ 65)\ \isanewline
\ \ \ \ \ \isakeyword{and}\ multiplication'\ (\isakeyword{infixl}\ "\isasymcdot '{\kern0pt}'"\ 70)\ \isakeyword{and}\ zero'\ ("\isasymzero '{\kern0pt}'")\ \isakeyword{and}\ unit'\ ("\isasymone '{\kern0pt}'")\isanewline
\ \ +\ \isakeyword{assumes}\ preimage\_of\_max\_lideal:\isanewline
\ \ \ \ \ \ "\isasymAnd \isasymww \isactrlsub A\ \isasymww \isactrlsub B.\ max\_lideal\ \isasymww \isactrlsub A\ A\ (+)\ (\isasymcdot )\ \isasymzero \ \isasymone \ \isasymLongrightarrow \ max\_lideal\ \isasymww \isactrlsub B\ B\ (+')\ (\isasymcdot ')\ \isasymzero '\ \isasymone '\isanewline
\ \ \ \ \ \ \ \ \ \ \ \isasymLongrightarrow \ (f\isactrlsup \isasyminverse \ A\ \isasymww \isactrlsub B)\ =\ \isasymww \isactrlsub A"
\end{isabelle}

Finally, we are able to introduce \emph{locally ringed spaces}.

\subsubsection*{Locally Ringed Spaces}

\begin{definition}[locally ringed space]
	A locally ringed space is a ringed space $(X, \mathscr{O}_X)$ such that the stalk $\mathscr{O}_{X, x}$ at $x$ is a local ring for every $x \in X$.
\end{definition}

\begin{isabelle}
\isacommand{locale}\ locally\_ringed\_space\ =\ ringed\_space\ +\isanewline
\ \ \isakeyword{assumes}\ stalks\_are\_local:\ "\isasymAnd x\ U.\ x\ \isasymin \ U\ \isasymLongrightarrow \ is\_open\ U\ \isasymLongrightarrow \ stalk.is\_local\ \isanewline
\ \ \ \ \ \ \ \ \ \ \ is\_open\ \isasymFF \ \isasymrho \ add\_str\ mult\_str\ zero\_str\ one\_str\ (neighborhoods\ x)\ x\ U"
\end{isabelle}

Moreover, one has the following result.

\begin{proposition}
	\label{prop:keyprop}
	For every $\mathfrak{p} \in \text{Spec}\,R$, the stalk $\mathscr{O}_{\spec, \mathfrak{p}}$ is isomorphic as a ring to $R_{\mathfrak{p}}$.	
\end{proposition}
\begin{proof}
	See \cite[Prop. 2.2.(a)]{hartshorne}.
\end{proof}	

\begin{isabelle}
\isacommand{locale}\ key\_map\ =\ comm\_ring\ +\isanewline
\ \ \isakeyword{fixes}\ \isasympp ::\ "'a\ set"\ \isakeyword{assumes}\ is\_prime:\ "\isasympp \ \isasymin \ Spec"\isanewline
\isakeyword{begin}\isanewline
\isanewline
\isacommand{interpretation}\ pi:pr\_ideal\ R\ \isasympp \ "(+)"\ "(\isasymcdot )"\ \isasymzero \ \isasymone 
\isanewline
\isanewline
\isacommand{interpretation}\ st:\ stalk\ "Spec"\ is\_zariski\_open\ sheaf\_spec\ sheaf\_spec\_morphisms\isanewline
\ \ \ \ \ \ \ \ \ \ \ \ \ \ \ \ \isasymO b\ add\_sheaf\_spec\ mult\_sheaf\_spec\ zero\_sheaf\_spec\ one\_sheaf\_spec\isanewline
\ \ \ \ \ \ \ \ \ \ \ \ \ \ \ \ "\{U.\ is\_zariski\_open\ U\ \isasymand \ \isasympp \isasymin U\}"\ \isasympp \isanewline
\isanewline
\isacommand{lemma}\ stalk\_at\_prime\_is\_iso\_to\_local\_ring\_at\_prime:\isanewline
\ \ \isakeyword{assumes}\ "is\_zariski\_open\ V"\ \isakeyword{and}\ "\isasympp \ \isasymin \ V"\isanewline
\ \ \isakeyword{shows}\ "\isasymexists \isasymphi .\ ring\_isomorphism\ \isasymphi \ st.carrier\_stalk\ st.add\_stalk\ st.mult\_stalk\ \isanewline
\ \ \ \ \ \ \ \ \ \ \ \ \ \ (st.zero\_stalk\ V)\ (st.one\_stalk\ V)\ (R\ \isactrlbsub \isasympp \ (+)\ (\isasymcdot )\ \isasymzero \isactrlesub )\ \isanewline
\ \ \ \ \ \ \ \ \ \ \ \ \ \ (pi.add\_local\_ring\_at)\ (pi.mult\_local\_ring\_at)\ \isanewline
\ \ \ \ \ \ \ \ \ \ \ \ \ \ (pi.zero\_local\_ring\_at)\ (pi.one\_local\_ring\_at)"\isanewline
\isanewline
\isacommand{end}
\end{isabelle}

Now, it suffices to teach Isabelle that a ring which is isomorphic to a local ring is local.

\begin{isabelle}
\isacommand{lemma}\ isomorphic\_to\_local\_is\_local:\isanewline
\ \ \isakeyword{assumes}\ "local\_ring\ B\ addB\ multB\ zeroB\ oneB"\isanewline
\ \ \ \ \isakeyword{and}\ "ring\_isomorphism\ f\ A\ addA\ multA\ zeroA\ oneA\ B\ addB\ multB\ zeroB\ oneB"\isanewline
\ \ \isakeyword{shows}\ "local\_ring\ A\ addA\ multA\ zeroA\ oneA"
\end{isabelle}
	
\begin{cor}
	$(\text{Spec}\,R, \mathscr{O}_{\spec} )$ is a locally ringed space.
\end{cor}

As a sanity check for our formal definitions, we eventually prove in Isabelle that the spectrum of a commutative ring is indeed a locally ringed space, a result that required some amount of formal machinery, especially to prove the key proposition \ref{prop:keyprop}.

\begin{isabelle}
\isacommand{lemma}\ (\isakeyword{in}\ comm\_ring)\ spec\_is\_locally\_ringed\_space:\isanewline
\ \ \isakeyword{shows}\ "locally\_ringed\_space\ Spec\ is\_zariski\_open\ sheaf\_spec\ \isanewline
\ \ \ \ \ \ \ \ \ \ \ \ \ \ sheaf\_spec\_morphisms\ \isasymO b\ add\_sheaf\_spec\ mult\_sheaf\_spec\ \isanewline
\ \ \ \ \ \ \ \ \ \ \ \ \ \ \ \ \ \ \ \ \ \ \ \ \ \ \ \ \ \ \ \ \ \ \ \ \ \ \ \ \ \ \ \ \ \ zero\_sheaf\_spec\ one\_sheaf\_spec"
\end{isabelle}

The required formal machinery and our formalization as a whole benefited from the interpretation mechanism in order to prevent statements and proof scripts from becoming too unwieldy. Without this mechanism the terms involved in the statements and proof scripts of such an intricate hierarchy of nested algebraic structures would have to be fed with many arguments, cluttering these statements and proof scripts to the point where they would become unreadable and difficult to use. Fortunately, the interpretation mechanism that comes with locales allows for reasonably concise and manageable terms with many implicit arguments which are declared beforehand using the interpretation mechanism and its command \textit{interpretation}.
  
Now, we introduce our last construction.
	Let $(X, \mathscr{O}_X)$, $(Y, \mathscr{O}_Y)$ be two ringed spaces and $(f, \phi_f)$ a morphism between them. Given $x \in X$, consider 
	\[
	I \coloneqq \lbrace V \mid V \, \text{is an open neighborhood of}\, f(x) \rbrace .
	\]
	Since we have a morphism $\phi_f: \mathscr{O}_Y \rightarrow f_* \mathscr{O}_X $ of sheaves on $Y$, we get a map between the direct limits over $I$. 
	\[
	\mathscr{O}_{Y, f(x)} \rightarrow \displaystyle \varinjlim_I f_{*} \mathscr{O}_X
	\]
	Last, there is a natural inclusion from $\displaystyle \varinjlim_I f_{*} \mathscr{O}_X$ to $\mathscr{O}_{X, x}$, since as $V$ ranges over $I$, $f^{-1}(V)$ ranges over a subset of all the open neighborhoods of $x$. So, eventually we get an induced morphism of rings from $\mathscr{O}_{Y, f(x)}$ to $\mathscr{O}_{X, x}$ which we will denote $\phi_{f, x}$.
As usual we embed this construction into a dedicated locale in Isabelle.

\begin{isabelle}
\isacommand{locale}\ ind\_mor\_btw\_stalks\ =\ morphism\_ringed\_spaces\ +\isanewline
\ \ \isakeyword{fixes}\ x::"'a"\isanewline
\ \ \isakeyword{assumes}\ is\_elem:\ "x\ \isasymin \ X"\isanewline
\isakeyword{begin}\isanewline
\isanewline
\isacommand{interpretation}\ stx:stalk\ X\ is\_open\isactrlsub X\ \isasymO \isactrlsub X\ \isasymrho \isactrlsub X\ b\ add\_str\isactrlsub X\ mult\_str\isactrlsub X\ zero\_str\isactrlsub X\ \isanewline
\ \ one\_str\isactrlsub X\ "\{U.\ is\_open\isactrlsub X\ U\ \isasymand \ x\ \isasymin \ U\}"\ "x"\isanewline
\isacommand{proof}\ \isacommand{qed}\ (auto\ simp:\ is\_elem)
\isanewline
\isanewline
\isacommand{interpretation}\ stfx:\ stalk\ Y\ is\_open\isactrlsub Y\ \isasymO \isactrlsub Y\ \isasymrho \isactrlsub Y\ d\ add\_str\isactrlsub Y\ mult\_str\isactrlsub Y\ zero\_str\isactrlsub Y\ \isanewline
\ \ one\_str\isactrlsub Y\ "\{U.\ is\_open\isactrlsub Y\ U\ \isasymand \ (f\ x)\ \isasymin \ U\}"\ "f\ x"\isanewline
\isacommand{proof}\ \isacommand{qed}\ (auto\ simp:\ is\_elem)
\isanewline
\isanewline
\isacommand{definition}\ induced\_morphism::\ "('c\ set\ \isasymtimes \ 'd)\ set\ \isasymRightarrow \ ('a\ set\ \isasymtimes \ 'b)\ set"\ \isanewline
\ \ \isakeyword{where}\ "induced\_morphism\ \isasymequiv \ \isasymlambda C\ \isasymin \ stfx.carrier\_stalk.\ \ \isanewline
\ \ \ \ \ \ \ \ \ \ \ \ \ \ let\ r\ =\ (SOME\ r.\ r\ \isasymin \ C)\ in\ \isanewline
\ \ \ \ \ \ \ \ \ \ \ \ \ \ \ \ \ \ \ \ \ \ stx.class\_of\ (f\isactrlsup \isasyminverse \ X\ (fst\ r))\ (\isasymphi \isactrlsub f\ (fst\ r)\ (snd\ r))"\isanewline
\isanewline
\isacommand{end}
\end{isabelle}

This last construction is required to define morphisms between locally ringed spaces.
			
\begin{definition}[morphism of locally ringed spaces]
	A morphism of locally ringed spaces from $(X, \mathscr{O}_X)$ to $(Y, \mathscr{O}_Y)$ is a morphism $(f, \phi_f)$ between locally ringed spaces such that the induced map of local rings $\phi_{f, x}: \mathscr{O}_{Y, f(x)} \rightarrow \mathscr{O}_{X, x}$ is a local homomorphism for every $x \in X$.  
\end{definition}

\begin{isabelle}
\isacommand{locale}\ morphism\_locally\_ringed\_spaces\ =\ morphism\_ringed\_spaces\ +\isanewline
\ \ \isakeyword{assumes}\ are\_local\_morphisms:\isanewline
\ \ \ \ "\isasymAnd x\ V.\ \isasymlbrakk x\ \isasymin \ X;\ is\_open\isactrlsub Y\ V;\ f\ x\ \isasymin \ V\isasymrbrakk \ \isasymLongrightarrow \ \isanewline
\ \ \ \ \ \ \ \ \ \ \ \ \ \ ind\_mor\_btw\_stalks.is\_local\ X\ is\_open\isactrlsub X\ \isasymO \isactrlsub X\ \isasymrho \isactrlsub X\ add\_str\isactrlsub X\ \isanewline
\ \ \ \ \ \ \ \ \ \ \ \ \ \ \ \ \ \ mult\_str\isactrlsub X\ zero\_str\isactrlsub X\ one\_str\isactrlsub X\ is\_open\isactrlsub Y\ \isasymO \isactrlsub Y\ \isasymrho \isactrlsub Y\ add\_str\isactrlsub Y\ \isanewline
\ \ \ \ \ \ \ \ \ \ \ \ \ \ \ \ \ \ \ \ \ \ mult\_str\isactrlsub Y\ zero\_str\isactrlsub Y\ one\_str\isactrlsub Y\ f\ x\ V\ \isasymphi \isactrlbsub X\ is\_open\isactrlsub X\ \isasymO \isactrlsub X\ \isanewline
\ \ \ \ \ \ \ \ \ \ \ \ \ \ \ \ \ \ \ \ \ \ \ \ \ \ \ \ \ \ \ \ \ \ \ \ \ \ \ \ \ \ \ \ \ \ \ \ \isasymrho \isactrlsub X\ is\_open\isactrlsub Y\ \isasymO \isactrlsub Y\ \isasymrho \isactrlsub Y\ f\ \isasymphi \isactrlsub f\ x\isactrlesub "
\end{isabelle}

\begin{remark}
	A morphism $(f, \phi_f)$ between locally ringed spaces is an isomorphism if and only if $f$ is an isomorphism between topological spaces, \textit{i.e.} a homeomorphism, and $\phi_f$ is an isomorphism between sheaves of rings.
\end{remark}

\begin{isabelle}
\isacommand{locale}\ iso\_locally\_ringed\_spaces\ =\ morphism\_locally\_ringed\_spaces\ +\isanewline
\ \ \isakeyword{assumes}\ is\_homeomorphism:\ "homeomorphism\ X\ is\_open\isactrlsub X\ Y\ is\_open\isactrlsub Y\ f"\ \isanewline
\ \ \ \ \ \ \isakeyword{and}\ is\_iso\_of\_sheaves:\ "iso\_sheaves\_of\_rings\ Y\ is\_open\isactrlsub Y\ \isasymO \isactrlsub Y\ \isasymrho \isactrlsub Y\ d\ \isanewline
\ \ \ \ \ \ \ \ \ \ \ \ add\_str\isactrlsub Y\ mult\_str\isactrlsub Y\ zero\_str\isactrlsub Y\ one\_str\isactrlsub Y\ im\_sheaf\ im\_sheaf\_morphisms\isanewline
\ \ \ \ \ \ \ \ \ \ \ \ \ \ \ b\ add\_im\_sheaf\ mult\_im\_sheaf\ zero\_im\_sheaf\ one\_im\_sheaf\ \isasymphi \isactrlsub f"
\end{isabelle}

Ultimately, we reach our goal: teaching schemes to Isabelle.

\subsubsection*{Schemes}	

\begin{definition}[affine scheme]
	An affine scheme is a locally ringed space $(X, \mathscr{O}_X)$ which is isomorphic (as a locally ringed space) to the spectrum $(\text{Spec}\,R, \mathscr{O}_{\spec} )$ for some commutative ring $R$. 
\end{definition}

\begin{isabelle}
\isacommand{locale}\ affine\_scheme\ =\ comm\_ring\ \isanewline
\ \ \ \ +\ locally\_ringed\_space\ X\ is\_open\ \isasymO \isactrlsub X\ \isasymrho \ b\ add\_str\ mult\_str\ \isanewline
\ \ \ \ \ \ \ \ \ \ \ \ \ \ \ \ \ \ \ \ \ \ \ \ \ \ \ \ \ \ \ \ \ \ \ \ \ \ \ \ \ \ \ \ \ \ \ \ \ \ \ \ zero\_str\ one\_str\ \isanewline
\ \ \ \ +\ iso\_locally\_ringed\_spaces\ X\ is\_open\ \isasymO \isactrlsub X\ \isasymrho \ b\ add\_str\ mult\_str\isanewline
\ \ \ \ \ \ \ \ zero\_str\ one\_str\ "Spec"\ is\_zariski\_open\ sheaf\_spec\ \isanewline
\ \ \ \ \ \ \ \ \ \ \ \ sheaf\_spec\_morphisms\ \isasymO b\ "\isasymlambda U.\ add\_sheaf\_spec\ U"\isanewline
\ \ \ \ \ \ \ \ \ \ \ \ \ \ \ \ \ \ "\isasymlambda U.\ mult\_sheaf\_spec\ U"\ "\isasymlambda U.\ zero\_sheaf\_spec\ U"\ \isanewline
\ \ \ \ \ \ \ \ \ \ \ \ \ \ \ \ \ \ \ \ \ \ \ \ \ \ \ \ \ \ \ \ \ \ \ \ \ \ \ \ \ \ "\isasymlambda U.\ one\_sheaf\_spec\ U"\ f\ \isasymphi \isactrlsub f\isanewline
\ \ \ \ \ \ \ \ \ \ \isakeyword{for}\ X\ is\_open\ \isasymO \isactrlsub X\ \isasymrho \ b\ add\_str\ mult\_str\ zero\_str\ one\_str\ f\ \isasymphi \isactrlsub f
\end{isabelle}

Of course, spectra of commutative rings being locally ringed spaces provide a class of affine schemes.

\begin{isabelle}
\isacommand{lemma}\ (\isakeyword{in}\ comm\_ring)\ spec\_is\_affine\_scheme:\isanewline
\ \ \isakeyword{shows}\ "affine\_scheme\ R\ (+)\ (\isasymcdot )\ \isasymzero \ \isasymone \ Spec\ is\_zariski\_open\ sheaf\_spec\ \isanewline
\ \ \ \ \ \ \ \ \ \ \ \ sheaf\_spec\_morphisms\ \isasymO b\ (\isasymlambda U.\ add\_sheaf\_spec\ U)\ \isanewline
\ \ \ \ \ \ \ \ \ \ \ \ \ \ (\isasymlambda U.\ mult\_sheaf\_spec\ U)\ (\isasymlambda U.\ zero\_sheaf\_spec\ U)\ \isanewline
\ \ \ \ \ \ \ \ \ \ \ \ \ \ \ \ \ \ \ \ \ \ \ \ (\isasymlambda U.\ one\_sheaf\_spec\ U)\ (identity\ Spec)\ \isanewline
\ \ \ \ \ \ \ \ \ \ \ \ \ \ \ \ \ \ \ \ \ \ \ \ \ \ \ \ \ \ \ \ \ \ \ \ \ \ \ \ \ \ \ \ \ \ (\isasymlambda U.\ identity\ (\isasymO \ U))"
\end{isabelle}

\begin{definition}[scheme]
	A scheme is a locally ringed space $(X, \mathscr{O}_X)$ in which every point has an open neighborhood $U$ such that the topological space $U$, together with the sheaf $\mathscr{O}_X | _U$, is an affine scheme.
\end{definition}

\begin{isabelle}
	\isacommand{locale}\isamarkupfalse%
	\ scheme\ {\isacharequal}{\kern0pt}\ locally{\isacharunderscore}{\kern0pt}ringed{\isacharunderscore}{\kern0pt}space\ X\ is{\isacharunderscore}{\kern0pt}open\ {\isasymO}\isactrlsub X\ {\isasymrho}\ b\ \isanewline
	\ \ \ \ \ \ \ \ \ \ \ \ \ \ \ \ \ \ \ \ \ \ \ \ \ \ add{\isacharunderscore}{\kern0pt}str\ mult{\isacharunderscore}{\kern0pt}str\ zero{\isacharunderscore}{\kern0pt}str\ one{\isacharunderscore}{\kern0pt}str\ \isanewline
	\ \ \isakeyword{for}\ X\ is{\isacharunderscore}{\kern0pt}open\ {\isasymO}\isactrlsub X\ {\isasymrho}\ b\ add{\isacharunderscore}{\kern0pt}str\ mult{\isacharunderscore}{\kern0pt}str\ zero{\isacharunderscore}{\kern0pt}str\ one{\isacharunderscore}{\kern0pt}str\ univ\ {\isacharplus}{\kern0pt}\isanewline
	\ \ \isakeyword{assumes}\ are{\isacharunderscore}{\kern0pt}affine{\isacharunderscore}{\kern0pt}schemes{\isacharcolon}{\kern0pt}\ {\isachardoublequoteopen}{\isasymAnd}x{\isachardot}{\kern0pt}\ x\ {\isasymin}\ X\ {\isasymLongrightarrow}\ \isanewline
	\ \ \ \ \ \ \ \ \ \ \ \ {\isacharparenleft}{\kern0pt}{\isasymexists}U{\isachardot}{\kern0pt}\ x{\isasymin}U\ {\isasymand}\ is{\isacharunderscore}{\kern0pt}open\ U\ {\isasymand}\ {\isacharparenleft}{\kern0pt}{\isasymexists}R\ add\ mult\ zero\ one\ f\ {\isasymphi}\isactrlsub f{\isachardot}{\kern0pt}\ \isanewline
	\ \ \ \ \ \ \ \ \ \ \ \ \ \ \ \ \ R\ {\isasymsubseteq}\ univ\ {\isasymand}\ comm{\isacharunderscore}{\kern0pt}ring\ R\ add\ mult\ zero\ one\ {\isasymand}\ \isanewline
	\ \ \ \ \ \ \ \ \ \ \ \ \ \ \ \ \ \ \ \ affine{\isacharunderscore}{\kern0pt}scheme\ R\ add\ mult\ zero\ one\ U\ \isanewline
	\ \ \ \ \ \ \ \ \ \ \ \ \ \ \ \ \ \ \ \ \ \ {\isacharparenleft}{\kern0pt}ind{\isacharunderscore}{\kern0pt}topology{\isachardot}{\kern0pt}ind{\isacharunderscore}{\kern0pt}is{\isacharunderscore}{\kern0pt}open\ X\ is{\isacharunderscore}{\kern0pt}open\ U{\isacharparenright}{\kern0pt}\ \isanewline
	\ \ \ \ \ \ \ \ \ \ \ \ \ \ \ \ \ \ \ \ \ \ {\isacharparenleft}{\kern0pt}ind{\isacharunderscore}{\kern0pt}sheaf{\isachardot}{\kern0pt}ind{\isacharunderscore}{\kern0pt}sheaf\ {\isasymO}\isactrlsub X\ U{\isacharparenright}{\kern0pt}\ \isanewline
	\ \ \ \ \ \ \ \ \ \ \ \ \ \ \ \ \ \ \ \ \ \ {\isacharparenleft}{\kern0pt}ind{\isacharunderscore}{\kern0pt}sheaf{\isachardot}{\kern0pt}ind{\isacharunderscore}{\kern0pt}ring{\isacharunderscore}{\kern0pt}morphisms\ {\isasymrho}\ U{\isacharparenright}{\kern0pt}\ b\ \isanewline
	\ \ \ \ \ \ \ \ \ \ \ \ \ \ \ \ \ \ \ \ \ \ {\isacharparenleft}{\kern0pt}ind{\isacharunderscore}{\kern0pt}sheaf{\isachardot}{\kern0pt}ind{\isacharunderscore}{\kern0pt}add{\isacharunderscore}{\kern0pt}str\ add{\isacharunderscore}{\kern0pt}str\ U{\isacharparenright}{\kern0pt}\isanewline
	\ \ \ \ \ \ \ \ \ \ \ \ \ \ \ \ \ \ \ \ \ \ {\isacharparenleft}{\kern0pt}ind{\isacharunderscore}{\kern0pt}sheaf{\isachardot}{\kern0pt}ind{\isacharunderscore}{\kern0pt}mult{\isacharunderscore}{\kern0pt}str\ mult{\isacharunderscore}{\kern0pt}str\ U{\isacharparenright}{\kern0pt}\ \isanewline
	\ \ \ \ \ \ \ \ \ \ \ \ \ \ \ \ \ \ \ \ \ \ {\isacharparenleft}{\kern0pt}ind{\isacharunderscore}{\kern0pt}sheaf{\isachardot}{\kern0pt}ind{\isacharunderscore}{\kern0pt}zero{\isacharunderscore}{\kern0pt}str\ zero{\isacharunderscore}{\kern0pt}str\ U{\isacharparenright}{\kern0pt}\isanewline
	\ \ \ \ \ \ \ \ \ \ \ \ \ \ \ \ \ \ \ \ \ \ {\isacharparenleft}{\kern0pt}ind{\isacharunderscore}{\kern0pt}sheaf{\isachardot}{\kern0pt}ind{\isacharunderscore}{\kern0pt}one{\isacharunderscore}{\kern0pt}str\ one{\isacharunderscore}{\kern0pt}str\ U{\isacharparenright}{\kern0pt}\ f\ {\isasymphi}\isactrlsub f{\isacharparenright}{\kern0pt}{\isacharparenright}{\kern0pt}{\isachardoublequoteclose}
\end{isabelle}

The only purpose of the variable \textit{univ} in the locale \textit{scheme} is to introduce properly a new type variable, in addition to the type variables from the carriers of $X$ and of the rings $\mathscr{O}_X(U)$, to be used in the type of $R$, the carrier of the ring which is introduced by an existential quantification in the assumption \textit{are\_affine\_schemes} once an open neighborhood $U$ of a given element $x$ in $X$ has been fixed. As the name \textit{univ} suggests, and as can be seen in the two lemmas below, this variable \textit{univ} should be instantiated with \textit{UNIV}, the set of all elements of a given type, while proving that a pair $(X, \mathscr{O}_X)$ is a scheme. The condition $R \subseteq \text{univ}$ in the locale \textit{scheme} is then trivially satisfied and it introduces no additional constraint other than a type constraint on the carrier of the ring $R$ as required by Isabelle type system. 

To match what was achieved using Lean~\cite{schemesinLean}, we finally prove in Isabelle that an affine scheme is a scheme and we give the example of the empty scheme $(\emptyset, \mathscr{O}_\emptyset)$, where $\mathscr{O}_\emptyset(\emptyset)$ is the zero ring $\lbrace 0 \rbrace$. 

\begin{isabelle}
	\isacommand{lemma}\isamarkupfalse%
	\ {\isacharparenleft}{\kern0pt}\isakeyword{in}\ affine{\isacharunderscore}{\kern0pt}scheme{\isacharparenright}{\kern0pt}\ affine{\isacharunderscore}{\kern0pt}scheme{\isacharunderscore}{\kern0pt}is{\isacharunderscore}{\kern0pt}scheme{\isacharcolon}{\kern0pt}\isanewline
	\ \ \isakeyword{shows}\ {\isachardoublequoteopen}scheme\ X\ is{\isacharunderscore}{\kern0pt}open\ {\isasymO}\isactrlsub X\ {\isasymrho}\ b\ add{\isacharunderscore}{\kern0pt}str\ mult{\isacharunderscore}{\kern0pt}str\ \isanewline
	\ \ \ \ \ \ \ \ \ \ \ \ \ \ \ \ \ \ \ \ \ \ \ \ \ \ \ \ zero{\isacharunderscore}{\kern0pt}str\ one{\isacharunderscore}{\kern0pt}str\ {\isacharparenleft}{\kern0pt}UNIV{\isacharcolon}{\kern0pt}{\isacharcolon}{\kern0pt}{\isacharprime}{\kern0pt}a\ set{\isacharparenright}{\kern0pt}{\isachardoublequoteclose}
\end{isabelle}

\begin{isabelle}
	\isacommand{lemma}\isamarkupfalse%
	\ empty{\isacharunderscore}{\kern0pt}scheme{\isacharunderscore}{\kern0pt}is{\isacharunderscore}{\kern0pt}scheme{\isacharcolon}{\kern0pt}\isanewline
	\ \ \isakeyword{shows}\ {\isachardoublequoteopen}scheme\ {\isacharbraceleft}{\kern0pt}{\isacharbraceright}{\kern0pt}\ {\isacharparenleft}{\kern0pt}{\isasymlambda}U{\isachardot}{\kern0pt}\ U{\isacharequal}{\kern0pt}{\isacharbraceleft}{\kern0pt}{\isacharbraceright}{\kern0pt}{\isacharparenright}{\kern0pt}\ {\isacharparenleft}{\kern0pt}{\isasymlambda}U{\isachardot}{\kern0pt}\ {\isacharbraceleft}{\kern0pt}{\isadigit{0}}{\isacharbraceright}{\kern0pt}{\isacharparenright}{\kern0pt}\ \isanewline
	\ \ \ \ \ \ \ \ \ \ \ \ {\isacharparenleft}{\kern0pt}{\isasymlambda}U\ V{\isachardot}{\kern0pt}\ identity{\isacharbraceleft}{\kern0pt}{\isadigit{0}}{\isacharcolon}{\kern0pt}{\isacharcolon}{\kern0pt}nat{\isacharbraceright}{\kern0pt}{\isacharparenright}{\kern0pt}\ {\isadigit{0}}\ {\isacharparenleft}{\kern0pt}{\isasymlambda}U\ x\ y{\isachardot}{\kern0pt}\ {\isadigit{0}}{\isacharparenright}{\kern0pt}\ \isanewline
	\ \ \ \ \ \ \ \ \ \ \ \ \ \ {\isacharparenleft}{\kern0pt}{\isasymlambda}U\ x\ y{\isachardot}{\kern0pt}\ {\isadigit{0}}{\isacharparenright}{\kern0pt}\ {\isacharparenleft}{\kern0pt}{\isasymlambda}U{\isachardot}{\kern0pt}\ {\isadigit{0}}{\isacharparenright}{\kern0pt}\ {\isacharparenleft}{\kern0pt}{\isasymlambda}U{\isachardot}{\kern0pt}\ {\isadigit{0}}{\isacharparenright}{\kern0pt}\ {\isacharparenleft}{\kern0pt}UNIV{\isacharcolon}{\kern0pt}{\isacharcolon}{\kern0pt}nat\ set{\isacharparenright}{\kern0pt}{\isachardoublequoteclose}
\end{isabelle}

\section{Concluding thoughts}\label{sec-conclusion}

\subsection*{Current Limitations of Locales}

Our wish to avoid Isabelle's records stemmed from their lack of multiple inheritance, as discussed above (\S\ref{subsec:locales}). While locales possess multiple inheritance, locales without records make our formalization unnecessarily verbose in some places. The mechanism for locale interpretation is currently missing a convenient way to bundle up a list of arguments in order to prevent argument lists from becoming unwieldy. Such a bundling would ideally still allow easy access to any component of a list. However, these practical limitations are not fundamental; they are engineering issues that could be tackled in the near future.   

\subsection*{The song of the Sirens}

It seems that our experience in Isabelle to formalize schemes was less tumultuous than the corresponding one in Lean, since Kevin Buzzard reported on his blog the difficulties he faced during his formalization in Lean:
\begin{quote}
	The project is completely incompatible with modern Lean and mathlib, but if you compile it then you get a sorry-free proof that an affine scheme is a scheme. During our proof we ran into a huge problem because our blueprint, the Stacks Project, assumed that $R[1/f][1/g]=R[1/fg]$, and this turns out to be unprovable for Lean’s version of equality [\dots]. The Lean community wrestled with this idea, and has ultimately come up with a beefed-up notion of equality for which the identity is now true.\footnote{\url{https://xenaproject.wordpress.com/2020/06/05/the-sphere-eversion-project/}}
\end{quote}	
This difficulty basically comes from defining the sheaf structure of $\spec$ first on the basis given by the so-called \textit{basic open sets} before extending it to all open subsets. We did not meet this difficulty in Isabelle, since we did not follow this sheaf-on-a-basis approach and the sheaves in our library are always defined from the get-go on all open subsets instead.
Also, Buzzard \textit{et al.}\ mention another difficulty encountered when proving that an affine scheme is a scheme. 
\begin{quote}
	This means that rewriting the equality $\iota(U) = U$ can cause technical problems with data-carrying types such as $\mathscr{O}_X(\iota(U))$ which depends on $\iota(U)$ (although in our case they were surmountable). This is a technical issue with dependent type theory and can sometimes indicate that one is working with the wrong definitions. \cite[3.5]{schemesinLean}
\end{quote}	
where $\iota: \spec \rightarrow \spec$ is here simply the identity map. If, again, we gave a slightly different formulation for the definition of a scheme, following Hartshorne~\cite{hartshorne} instead of the Stacks project \cite{stacksproject} as they did, it seems that this difficulty cannot possibly have a direct counterpart in Isabelle, since it arose from the difference between the so-called \textit{equality types}, also known as \textit{identity types}, and the \textit{definitional equality}, a difference which is peculiar to dependent type theories. This difficulty was eventually overcome in Lean using a trick.  
\begin{quote}
	Fortunately, in our case, Reid Barton pointed out the following extraordinary trick to us: we can define the map $\mathscr{O}_X(\iota(U)) \rightarrow \mathscr{O}_X(U)$ using restriction rather than trying to force it to be the identity! \cite[3.5.]{schemesinLean}
\end{quote}	    
The gap between the type theories of Lean and Isabelle/HOL called for the very different approach of the present work. If our approach based on Isabelle's locales required some craftsmanship, locales offered enough flexibility to avoid dead ends. While dependent types are expressive, expert users know they should be used with parsimony and some issues have been highlighted for instance in the context of formalising analysis in the proof assistant Coq \cite[p.42]{boldo2015coquelicot}.
Moreover, in the age of machine learning, one may expect there is a trade-off between the expressiveness of a type theory and one's ability to subject it to automation, for instance through SMT solvers and the so-called hammers embedded in some proof assistants. How deep one can go formalizing mathematics in simple type theory? This work provides a first hint at an answer that should not be ruled out: simple type theory is not too simple.        

\section*{Funding}

This work was supported by the ERC Advanced Grant ALEXANDRIA (Project GA 742178).

\section*{Acknowledgments}
 
We thank Kevin Buzzard for his stimulating wit and his communicative energy and Andr\'e Hirschowitz who commented on a draft of this document. We also thank the reviewers, one of whom noticed an inaccuracy (fortunately easy to correct) in the formal translation into Isabelle of Definition 3.21.  			

\bibliographystyle{plain}
\bibliography{ALEXANDRIA}

\end{document}